\documentclass[absolute,12pt]{mymathart}
\usepackage{mymathmacros}
\usepackage{color}

 \newtheoremstyle{numberedstyle}
   {9pt}
   {9pt}
   {\normalfont}
   {}
   {\bfseries}
   {.}
   {\newline}
   {}

\newcommand{\gconj}{\mathfrak{g}}

\numberwithin{equation}{section}

\newcommand{\addR}{\underline{r}}

\theoremstyle{changebreak}
\newtheorem{thm}{Theorem}[section]%
\newtheorem{lem}[thm]{Lemma}%
\newtheorem{cor}[thm]{Corollary}%
\newtheorem{prop}[thm]{Proposition}%
\newtheorem{conj}[thm]{Conjecture}%
\newtheorem{obs}[thm]{Observation}%

\theoremstyle{numberedstyle}
\newtheorem{rmk}[thm]{Remark}%
\newtheorem{question}[thm]{Question}%
\newtheorem{defn}[thm]{Definition}%

\usepackage{hyperref}

\title{Escaping endpoints explode}
\author{Nada Alhabib}
\address{Dept.\ of Mathematical Sciences, University of Liverpool, Liverpool L69 7ZL, UK.}
\email{n.a.alhabib@liverpool.ac.uk}
\author{Lasse Rempe-Gillen}
\address{Dept.\ of Mathematical Sciences, University of Liverpool, Liverpool L69 7ZL, UK. ORCiD: 0000-0001-8032-8580.}
\email{l.rempe@liverpool.ac.uk}
\thanks{The first author thanks
  Princess Nora bint Abdulrahman University in Saudi Arabia
   for supporting her doctoral studies at the University of Liverpool. The second author is supported by a Philip Leverhulme Prize.}
\subjclass[2010]{Primary 37F10; Secondary 30D05, 54F15, 54G15}

\newcommand{\F}{\mathcal{F}}

\newcommand{\s}{\underline{s}}
\newcommand{\addm}{\underline{m}}
\newcommand{\raddr}{\underline{r}}
\newcommand{\K}{\mathbb{K}}

\newcommand{\extaddr}{\operatorname{addr}}

\newcommand{\ts}{t_{\s}}

\newcommand{\itin}{\operatorname{itin}}
\newcommand{\addu}{\underline{u}}
\newcommand{\addt}{\underline{t}}
\newcommand{\Sequ}{\mathbb{S}}
\newcommand{\Sequb}{\overline{\Sequ}}

\begin{document} 

\begin{abstract}
  In 1988, Mayer proved the remarkable fact that $\infty$ is an explosion point for
    the set $E(f_a)$ of endpoints of the Julia set of $f_a\colon\C\to\C; e^z+a$ with $a<-1$. That is, the set
    $E(f_a)$ is totally separated (in particular, it does not have any nontrivial connected subsets), 
    but $E(f_a)\cup\{\infty\}$ is connected. Answering a question of Schleicher,
    we extend this result to the set $\tilde{E}(f_a)$ of \emph{escaping endpoints} in the sense of Schleicher and Zimmer,
    for any parameter $a\in\C$ for which the singular value $a$ belongs to an attracting or parabolic basin, has a finite orbit, or
    escapes to infinity under iteration (as well as many other classes of parameters). 

    Furthermore, we extend one direction of the theorem to much greater generality, by proving that the set $\tilde{E}(f)\cup\{\infty\}$ is connected
     for any transcendental entire function $f$ of finite order with bounded singular set. 
    We also discuss corresponding results for \emph{all} endpoints in the case of exponential maps; in order to do so, we establish 
     a version of Thurston's \emph{no wandering triangles} theorem for exponential maps. 
\end{abstract}

\maketitle

\section{Introduction}
   
 A point $x_0$ is called an \emph{explosion point} of a metric space $X$ if 
   $X$ is connected but $X\setminus\{x_0\}$ is \emph{totally separated}. The latter means that,
    for any $a,b\in X\setminus \{x_0\}$, there are open sets 
    $U_a,U_b\subset X$ such that $a\in U_a$, $b\in U_b$,
    $U_a\cup U_b=X$ and $U_a\cap U_b = \{x_0\}$. If $X$ is connected but $X\setminus\{x_0\}$ is 
    totally disconnected, then $x_0$ is called a \emph{dispersion point} of $X$. 

    In this article, we will typically consider the case of $X=A\cup \{\infty\}$, where $A\subset \C$, and
    $\infty$ is an explosion/dispersion point of $X$. In a slight abuse of terminology, we shall then also say that 
     $\infty$ is an explosion/dispersion point for $A$.

Clearly every explosion point is also a
    dispersion point, but the converse is not true. Indeed,   
   in 1921, Knaster and Kuratowski first gave an example of a space having a dispersion point, now known as the \emph{Knaster-Kuratowski fan} 
   \cite[Exemple~$\alpha$, \S5]{knasterkuratowski}. This fan does not have an explosion point; note that
    in  particular this gives an example of a space that is totally disconnected but not totally separated.
  Wilder~\cite{wilderexplosion} constructed
    the first example of a space having an explosion point shortly afterwards, in 1923. Another famous example of a space having an explosion point, which 
    is of key importance for this paper, is the set of endpoints of the \emph{Lelek fan} (see Section~\ref{sec:topology}), constructed by Lelek 
     \cite[\S 9]{lelekfan} in 1961. 
    It is tempting to think of such spaces as 
   ``pathological''. However, Mayer \cite{mayerexplosion}
   showed that explosion points occur very naturally within the iteration 
   of transcendental entire functions. 

   To explain his result, recall that the \emph{Julia set} $J(f)$ of a transcendental entire function $f$ consists of the
   points at which the family $(f^n)_{n=1}^{\infty}$ of the iterates of $f$ is not equicontinuous with respect to the spherical metric (in other words,
   this is where the function behaves ``chaotically''). Its complement $F(f)\defeq \C\setminus J(f)$ is called the \emph{Fatou set}. 
   The iteration of transcendental entire functions has enjoyed significant interest recently.
    The family of quadratic polynomials $z\mapsto  z^2+c$, which gives rise  to the 
      famous \emph{Mandelbrot set}, has been studied intensively as a prototype of
      polynomial and rational dynamics. Similarly, complex \emph{exponential maps} 
     \begin{equation}\label{eqn:exponentialfamily}
       f_a\colon \C\to \C; \qquad z\mapsto e^z+a,
     \end{equation}
   form the simplest space of transcendental entire functions. (Compare \cite{dghnew2,expbifurcationlocus} for a discussion of connections between the  two  families.) Hence there is 
   considerable interest in understanding the fine structure of the Julia sets of the  functions $f_a$.

  In the case where $a\in (-\infty, -1)$, it is known \cite[Theorem on p.~50]{devaneykrych} that $J(f_a)$ consists of uncountably
   many curves, each connecting a finite \emph{endpoint} to infinity. 
   It is easy to see that this set $E(f_a)$ of endpoints is totally separated; Mayer proved that 
   $E(f_a)\cup\{\infty\}$ is connected, and hence has $\infty$ as an explosion point. In fact, by work of
    Aarts and Oversteegen \cite{aartsoversteegen}, $J(f_a)\cup\{\infty\}$ is homeomorphic to the Lelek fan for these parameters.

   For general $a\in\C$, the Julia set no longer has such a simple structure. Nonetheless, there is still a natural notion of
    ``rays'' or ``hairs'', generalising the curves mentioned above. While some of these may no longer land at finite endpoints
     \cite{devaneyjarqueindecomposable,nonlanding}, Schleicher and Zimmer \cite{expescaping} have shown that there is always a set 
     $\tilde{E}(f_a)$ of \emph{escaping endpoints}, which moreover has a certain universal combinatorial structure that is independent of the
     parameter $a$, making it a natural object for dynamical considerations. For the purpose of this introduction, we define the
     relevant concepts as follows; see Section~\ref{sec:exponentialgeneral}, and in particular Corollary~\ref{cor:endpointcharacterisation},  for further discussion. 

  \begin{defn}[Types of escaping points]\label{defn:endpoints}
   The set of \emph{escaping points} of $f_a$ is denoted by
     \[ I(f_a) \defeq \{z\in\C\colon f_a^n(z)\to\infty\}. \]
   We say that a point $z_0\in\C$ is \emph{on a hair} if there exists
    an arc $\gamma\colon[-1,1]\to I(f_a)$ such that $\gamma(0)=z_0$. 

   We say that a point $z_0\in\C$ is an \emph{endpoint} if $z_0$ is
    not on a hair and there is an arc $\gamma\colon[0,1]\to \C$ such that
    $\gamma(0)=z_0$ and $\gamma(t)\in I(f_a)$ for all $t>0$.
    The set of all endpoints is denoted by $E(f_a)$, while
    $\tilde{E}(f_a) \defeq E(f_a)\cap I(f_a)$ denotes the set of \emph{escaping endpoints}. 
  \end{defn}

 With this terminology, Schleicher (personal communication) asked whether Mayer's phenomenon can be extended to the set of escaping endpoints
   for \emph{all} exponential maps.

\begin{question}[Schleicher]\label{question:schleicher}
  Let $a\in\C$. Is $\infty$ an explosion point for $\tilde{E}(f_a)$? 
\end{question}

 In this article, we provide a positive answer to this question for a large class of parameters $a$. 
  \begin{thm}[Escaping endpoints explode]\label{thm:escapingendpointsexplode}
    Let $a\in\C$, and suppose that the singular value $a$ of the map $f_a$ satisfies one of the following conditions.
   \begin{enumerate}[(a)]
    \item $a$ belongs to the Fatou set;
    \item $a$ is on a hair;
    \item $a$ is an endpoint.   
   \end{enumerate}
  Then $\infty$ is an explosion point for $\tilde{E}(f_a)$. 
  \end{thm}
 \begin{remark}
  The hypotheses  of the theorem are quite natural and general, 
   and have appeared 
   previously in the study of exponential maps 
   \cite{nonlanding,escapingconnected_2}. They are 
   known to hold, in particular,
   whenever $f_a$ has an attracting or parabolic orbit and when
   the singular value $a$ escapes to infinity \cite{expescaping}, 
   is preperiodic \cite{expper} or nonrecurrent \cite{beninilyubich}.
   It is likely that there are some parameters $a\in\C$ for which the hypotheses fail; compare Section~\ref{sec:further}. However, the 
   hypotheses can be weakened to a condition that, conjecturally, holds whenever $f_a$ does not have an irrationally indifferent periodic point; see 
    Theorem~\ref{thm:ghostlimbs}. 
 \end{remark}

 By definition, the proof of Theorem~\ref{thm:escapingendpointsexplode} requires two steps: 
   On the one hand, we must show that $\tilde{E}(f_a)\cup\{\infty\}$ is connected, and on the other that
    $\tilde{E}(f_a)$ is totally separated. It turns out that the first can be carried out in complete generality.

 \begin{thm}[Escaping endpoints connect to infinity] \label{thm:escapingendpoints}
  The set    $\tilde{E}(f_a)\cup\{\infty\}$ is connected for all $a\in\C$.
 \end{thm}
\begin{remark}
  Our proof shows that $\tilde{E}(f_a)$ can be replaced by the set $\tilde{E}(f_a)\cap A(f_a)$ of all endpoints
    belonging to the \emph{fast escaping set}; see Remark~\ref{rmk:fast}. (The set $A(f_a)$ consists of all points that escape to infinity
     ``as fast as possible'', and has played an important role in recent results of transcendental dynamics; compare~\cite{ripponstallardfast}.)
\end{remark}

 We remark that this result is new even for $a\in (-\infty,-1)$. Indeed, our strategy of proof is to establish the theorem first in this setting,
   using a topological model introduced in~\cite{topescaping}, and then infer the general case using a conjugacy result also
   proved in~\cite{topescaping}. 

 In the opposite direction, under the hypotheses of Theorem~\ref{thm:escapingendpointsexplode}, we are able to
  employ combinatorial techniques to conclude that $\tilde{E}(f_a)$ is indeed totally separated. Without such hypotheses, 
   the answer to Question~\ref{question:schleicher} turns out to depend on extremely difficult problems concerning the possible accumulation behaviour of
  rays in the dynamical plane. For example, it is widely believed that a path-connected component 
  of the escaping set of an exponential map cannot be dense in the plane
   \cite[Question 7.1.33]{thesis}. (The nature of the path-connected components of $I(f_a)$ is well-known: each such component~-- with the 
    exception, in certain well-understood cases, of countably many additional curves~--  is one of the 
    afore-mentioned dynamics rays. See
 \cite[Corollary~4.3]{frs}.) However, 
   the corresponding question (whether an external ray can ever accumulate on the whole Julia set) remains open, as far as we are aware, even for
  quadratic polynomials. The following result shows that Question~\ref{question:schleicher} 
  is at least as difficult. 

 \begin{thm}[Separating escaping endpoints is hard]\label{thm:denseray}
  Suppose that, for some parameter $a\in \C$, the set 
   $I(f_a)$ has a path-connected component that is dense in the Julia set. 
   Then $\tilde{E}(f_a)$ is connected. 
 \end{thm} 

However, as we discuss in Section~\ref{sec:further}, it appears quite likely that the answer to Question~\ref{question:schleicher} is always positive.

 \begin{conj}[Escaping endpoints always explode] \label{conj:endpoints}
  Let $a\in \C$. Then $\tilde{E}(f_a)$ is totally separated
   (and hence 
     $\infty$ is an explosion point of $\tilde{E}(f_a)\cup\{\infty\}$.) 
 \end{conj}

\subsection*{The set of all endpoints}
  The set $E(f_a)\supset\tilde{E}(f_a)$ of \emph{all} endpoints of an exponential map $f_a$ is a less universal object than that of escaping endpoints, since
    its structure may vary considerably with the parameter $a$. Nonetheless, one can ask the analogue of Question~\ref{question:schleicher} for this 
    set. The set $\tilde{E}(f_a)$ of escaping endpoints is completely invariant under $f_a$, and hence
    dense in the Julia set. Thus, by Theorem~\ref{thm:escapingendpoints}, the set $E(f_a)\cup\{\infty\}$ contains a dense connected
   subset for all $a\in\C$, and is therefore connected. In the opposite
    direction, we are able to prove the following.
  (See Section~\ref{sec:accessible} for the definition
    of the kneading sequence). 

 \begin{thm}[Endpoints disperse] \label{thm:endpoints}
    Let $a\in\C$, and suppose that the singular value $a$ of the function $f_a$ satisfies one of the following conditions.
   \begin{enumerate}[(a)]
    \item $a$ belongs to the Fatou set;\label{item:fatou}
    \item $a$ is an endpoint or on a hair, and its kneading sequence is non-periodic.
   \end{enumerate}
  Then $\infty$ is a dispersion point for $E(f_a)$; in the case of~\ref{item:fatou}, $\infty$ is also an explosion point. 
  \end{thm}

 The proof of Theorem~\ref{thm:endpoints} requires a more detailed combinatorial analysis than that of Theorem~\ref{thm:escapingendpointsexplode}.
   In particular, we establish a combinatorial version of Thurston's \emph{no wandering triangles} theorem for exponential maps
   (Theorem~\ref{thm:NWT}). The following consequence of the latter result may be of independent interest. 

 \begin{thm}[No wandering triods]\label{thm:nowanderingtriods}
   Let $f_a$ be an exponential map. Suppose that $z_0$ is an endpoint of $f_a$ that is accessible 
  from $I(f_a)$ in at least three different ways; that is, 
  there are three arcs in $I(f_a)\cup\{z_0\}$ that contain $z_0$ and are otherwise pairwise 
  disjoint. 

    Then $z_0$ is eventually periodic. 
 \end{thm}

\subsection*{More general functions}
  In this article, we focus primarily on the family of exponential maps. However, there is a large class of
   transcendental entire functions for which the existence of rays, in analogy to the exponential family, was established in~\cite{rrrs} (and, independently,
  in \cite{baranskihyperbolic} for a more restrictive class of maps).
   It makes sense to ask about the properties of the set of escaping endpoints in this context also. In this direction, we shall prove the following
   extension of Theorem~\ref{thm:escapingendpoints}. 
  (Recall that an entire function has \emph{finite order} if $\log\log|f(z)| = O(\log|z|)$ as $z\to\infty$, and that the
    set $S(f)$ of \emph{singular values} is the closure of the set of critical and asymptotic values of $f$. For further background, we refer e.g.\ 
    to~\cite{brush}.) 

\begin{thm}[Escaping endpoints of functions with rays]\label{thm:general}
  Let $f$ be a finite-order entire function with bounded singular value set $S(f)$, 
    a finite composition of such functions, or more generally a function satisfying a ``uniform head-start condition'' in the sense of \cite{rrrs}. 
    Then the set $\tilde{E}(f)\cup\{\infty\}$ is connected. 

   If, additionally, $f$ is hyperbolic~-- that is, every singular value $s\in S(f)$ belongs to the basin of an attracting periodic orbit~--
     then $\infty$ is an explosion point for $\tilde{E}(f)$. 
\end{thm}

 \subsection*{Idea of the proofs}
   Establishing Theorem~\ref{thm:escapingendpointsexplode} for  
    the case of $a\in (-\infty,-1)$ is the key step in our arguments. We shall
   do so using an explicit topological model for the dynamics of such $f_a$ on its Julia set, introduced in \cite{topescaping}.
   We review its definition and properties in Section~\ref{sec:model}.
   and then prove Theorem \ref{thm:escapingendpoints} in this case by
   showing that $J(f_a)$ contains an invariant
   subset $A$ such that $A\cup \{\infty\}$ is homeomorphic to the
     Lelek fan, and such that every endpoint of $A$ is an escaping endpoint
   of $f_a$. 

  We are able to transfer this result to the dynamical plane of any
   exponential map using a general conjugacy result from \cite{topescaping} (Theorem~\ref{thm:boettcher}).
   To show that $\tilde{E}(f_a)$ is totally separated in the cases covered by
   Theorem \ref{thm:escapingendpointsexplode}, we review and apply combinatorial
   methods for exponential maps in Section \ref{sec:accessible}. 
    Theorem~\ref{thm:endpoints} is also proved there, using similar arguments but requiring the \emph{no wandering triangles theorem}. 
   Finally, we consider the case of more general entire functions in
   Section~\ref{sec:general} and discuss further questions, including
    Conjecture \ref{conj:endpoints}, in Section
   \ref{sec:further}.

  \subsection*{Basic notation}
    As usual, $\C$ denotes the complex plane; its one-point compactification is the
    \emph{Riemann sphere} $\Ch\defeq \C\cup\{\infty\}$. 
    The set of integer sequences $\s=s_0 s_1 s_2 \dots$ is denoted by $\Z^{\N_0}$;
      if $s_0 \dots s_{n-1}$ is a finite sequence of integer, then its periodic extension is denoted
       $(s_0 \dots s_{n-1})^{\infty}\in \Z^{\N_0}$.

   Throughout the article (with the exception of Section~\ref{sec:general}), we
     consider the family $(f_a)_{a\in\C}$ of exponential maps defined by \eqref{eqn:exponentialfamily}. 
    This family is also often parameterised as $(w\mapsto \lambda e^w)_{\lambda\in\C\setminus\{0\}}$; the two families are conjugate via 
    the translation $z=w+a$, with $\lambda = e^a$. Our choice of parameterisation gives a slightly more convenient asymptotic description
    of \emph{dynamic rays} (see the remarks on notation in \cite[p.~108]{bifurcations_new}, and also the discussion in Section~2.3 of \cite{thesis}), 
     but as we shall not require this additional information,
    the choice is primarily a matter of taste.

 \subsection*{Acknowledgements} 
  We thank Anna Benini, Bob Devaney, Alexandre DeZotti, Vasiliki Evdoridou, Dierk Schleicher and Stephen Worsley for interesting discussions.
     We are also very grateful to the referee for many thoughtful comments that have helped us to improve
     the presentation of the paper.

\section{Lelek fans, straight brushes and Cantor bouquets} \label{sec:topology}
  
 We begin by formally introducing the topological notions central to this paper
    (see e.g.\ \cite[Section~4]{counterexamplestopology}).   
 \begin{defn}[Separation]
    Let $X$ be a topological space, and let $x_0\in X$. 
     \begin{enumerate} 
       \item Two points $a,b\in X$ are \emph{separated} (in $X$) if there is an open and closed subset $U\subset X$ with $a\in U$ and $b\notin U$. 
       \item The space $X$ is called \emph{totally separated} if any two points are separated in $X$.
       \item The space $X$ is called \emph{totally disconnected} if it contains no nontrivial connected subset. 
       \item If $X$ is connected and $X\setminus\{x_0\}$ is totally separated, then $x_0$ is called an \emph{explosion point} of $X$.
       \item If $X$ is connected and $X\setminus\{x_0\}$ is totally disconnected, then $x_0$ is called a \emph{dispersion point} of $X$. 
     \end{enumerate}
 \end{defn}

 The following~-- a version of the well-known fact that the continuous image of a connected space is connected~-- is immediate from the definitions.

\begin{obs}[Preimages of separated points are separated]\label{obs:preimageseparation}
  Let $f\colon X\to Y$ be a continuous map between topological spaces, and let $a,b\in X$. If $f(a)$ and $f(b)$ are separated in $Y$, then
    $a$ and $b$ are separated in $X$.
\end{obs}
\begin{proof}
  If $U$ is open and closed in $Y$ with $f(a)\in U$ and $f(b)\notin U$, then $f^{-1}(U)$ is open and closed in $X$ with
    $a\in f^{-1}(U)$ and $b\notin f^{-1}(U)$. 
\end{proof}

   The sets we are usually interested in are subsets of the plane.  
    Here the notion of separation takes the following, particularly simple, form.
    (See e.g.\ \cite[Lemma~3.1]{escapingconnected_2}.)

 \begin{lem}[Separation in the plane]\label{lem:planeseparation}
    Let $X\subset\C$, and let $x,y\in X$. Then $x$ and $y$ are separated in $X$ if and only if there is a closed and connected set $\Delta\subset\C\setminus X$ 
     such that $x$ and $y$ belong to different connected components of $\C\setminus \Delta$. 
  \end{lem}

\subsection*{The Lelek fan}
 As mentioned in the introduction, the \emph{Lelek fan} will play an important role in our discussions. It can be defined as follows. 

\begin{defn}[Lelek fans]\label{defn:lelekfan}
  Let $X$ be a \emph{continuum}, i.e.\ a compact, connected metric space. 
    Then $X$ is a \emph{fan} (with \emph{top} $x_0\in X$) if the following conditions are satisfied:
    \begin{enumerate}[(a)]
      \item $X$ is \emph{hereditarily unicoherent}; that is, if $A,B\subset X$ are subcontinua of $X$, then $A\cap B$ is connected.\label{item:unicoherent}
      \item $X$ is \emph{arcwise connected}. Hence, by ~\ref{item:unicoherent}, $X$ is \emph{uniquely} arcwise connected, and we use
        $[x,y]$ to denote the arc connecting $x,y\in X$. 
      \item $x_0$ is the only \emph{ramification point} of $X$ (that is, a common endpoint of at least three different arcs that are otherwise pairwise disjoint). 
    \end{enumerate}
   If $x\in X$ is an endpoint of every arc containing it, then $x$ is called an \emph{endpoint} of $X$; the set of all endpoints is denoted $E(X)$.

   A fan $X$ with top $x_0$ is called \emph{smooth} if, for any sequence $y_n$ converging to a point $y$, the arcs 
      $[x_0,y_n]$ converge to $[x_0,y]$ in the Hausdorff metric. 

   A smooth fan such that the endpoints of $X$ are dense in $X$ is called a \emph{Lelek fan}. 
\end{defn}

 In \cite[\S 9]{lelekfan}, Lelek gave an example of a fan $X$ with the above properties, and showed that 
    $E(X)\cup\{x_0\}$ has $x_0$ as an explosion point. Our terminology above is justified by the fact, proved independently in 
   \cite[Corollary on p.\ 33]{charatoniklelekfan} and in
   \cite{bulaoversteegen}, that any smooth fan with a dense set of endpoints is homeomorphic to that constructed by Lelek:

 \begin{prop}[Lelek fans and explosion points]\label{prop:lelekfan}
  \begin{enumerate}[(a)]
    \item
    Any two Lelek fans are homeomorphic.
     \item If $X$ is a Lelek fan and $x_0$ is the top of $X$, then $x_0$ is an explosion point for
    $E(X)\cup \{x_0\}$. 
   \end{enumerate}
 \end{prop}

\subsection*{Straight brushes and Cantor Bouquets}
  We will encounter a variety of Lelek fans that are embedded in the Riemann sphere (with top $\infty$). However, there are several inequivalent
   embeddings of the Lelek fan, corresponding to the fact that some of the ``hairs'' (connected components of the complement of the top) might be
   approximated on
   by other hairs only from one side. In order to identify a preferred embedding, where this does not occur, Aarts and Oversteegen introduced the
   following terminology. (We remark that the concepts in this subsection are required only for the proof of Theorem~\ref{thm:denseray}.) 

 \begin{defn}[Straight brush {\cite[Definition~1.2]{aartsoversteegen}}] \label{defn:straightbrush}
    A \emph{straight brush} $B$ is a subset of $\{(y,\alpha)\in\R^2\colon y\geq 0\text{ and }\alpha\notin\Q\}$ with the following properties.
    \begin{description}
       \item[Hairiness] 
           For every $\alpha\in\R$, there is a $t_{\alpha}\in [0,\infty]$ such that 
             $(t,\alpha)\in B$ if and only if $t\geq t_{\alpha}$. 
      \item[Density] The set of $\alpha$ with $t_{\alpha}<\infty$ is dense in $\R$. Furthermore, for every
          such $\alpha$, there exist sequences $(\beta_n)$ and $(\gamma_n)$ such that $\beta_n \nearrow \alpha$ and $\gamma_n \searrow\alpha$, 
          and such that $t_{\beta_n}, t_{\gamma_n}\to t_{\alpha}$. 
     \item[Compact sections] $B$ is a closed subset of $\R^2$. 
    \end{description}
\end{defn}

  If $B$ is a straight brush, then clearly   
   its one-point compactification $B\cup\{\infty\}$ is a smooth fan. Furthermore, since $B$ is closed, the function
   $\alpha\mapsto t_{\alpha}$ is lower semicontinuous, and it follows easily that the set of endpoints is dense in $B$
   (compare~\cite[Corollary~2.5]{aartsoversteegen}). Hence $B\cup\{\infty\}$ is a Lelek fan.

 \begin{defn}[Cantor Bouquet {\cite[Definition~1.1]{brush}}]\label{defn:cantorbouquet}
  A closed set $A\subset\C$ is called a \emph{Cantor bouquet} if it is ambiently homeomorphic to a straight brush $B$; i.e., if there is
    a homeomorphism $\phi\colon\C\to \R^2$ such that $\phi(A)=B$. 
 \end{defn}
 \begin{remark}
   The term ``Cantor bouquet'' has been used informally since the 1980s; in particular,    
   Devaney and Tangerman~\cite[p.~491]{devaneytangerman} give 
    a definition of Cantor bouquets that differs from the one above. Definition~\ref{defn:cantorbouquet} first appears
     in~\cite[Definition~3.3]{devaneybouquets}. (We note that, in \cite{devaneybouquets}, 
     the homeomorphism is not required to be ambient, but this appears to have been an oversight.)
  \end{remark}

  The following characterisation of Cantor Bouquets, in the spirit of Definition~\ref{defn:lelekfan}, will be useful, although strictly speaking
   we shall not require it.

 \begin{thm}[Characterisation of Cantor bouquets]\label{thm:cantorbouquets}
    A closed set $A\subset\C$ is a Cantor bouquet if and only if
     \begin{enumerate}[(a)]
       \item $A\cup\{\infty\}$ is a Lelek fan with top $\infty$, and \label{item:lelekfan}
       \item if $x\in A$ is accessible from $\C\setminus A$, then $x$ is an endpoint of $A$. 
         (Equivalently, every hair of $A$ is accumulated on by other hairs from both sides.)\label{item:bothsides} 
     \end{enumerate}

   Moreover, if $\phi$ is any homeomorphism between Cantor bouquets that preserves the cyclic order of hairs at $\infty$, then
     $\phi$ extends to an orientation-preserving homeomorphism $\C\to\C$. 
 \end{thm}
 \begin{remark}
    To explain the hypothesis of the final statement, observe that 
     any family of pairwise disjoint arcs in $\C$, each having one finite endpoint and another endpoint at infinity, induces a natural \emph{cyclic order}.
      By this we mean that there is a consistent notion of whether three such arcs $\gamma_1,\gamma_2,\gamma_3$ are ordered in positive or
      negative orientation. 
      Indeed, if $\gamma_1$ and $\gamma_2$ are such arcs, then for sufficiently large $R$ the set 
      $\C\setminus (\overline{D_R(0)}\cup\gamma_1\cup \gamma_2)$ has precisely two unbounded connected components, one that lies between
      $\gamma_1$ and $\gamma_2$ (in positive orientation) and one that lies between $\gamma_2$ and $\gamma_1$. If $\gamma_3$ tends to
      infinity in the latter, then the triple is ordered in positive orientation. It is easy to see that this does indeed define
      a cyclic order. 

    In our applications, all curves will have real parts tending to $+\infty$, and in this case the cyclic order can be upgraded to a
     linear order, referred to as the \emph{vertical order}. (The curve $\gamma_1$ is \emph{below} $\gamma_2$ if $\gamma_1,\gamma_2,\alpha$ are
     ordered in positive orientation, where $\alpha$ is a horizontal line segment tending to $-\infty$. 
  \end{remark}
  \begin{proof}[Sketch of proof]
    Let us call a closed set $A\subset \C$ a \emph{candidate bouquet} if it satisfies~\ref{item:lelekfan} 
      and~\ref{item:bothsides}. Note  that any Cantor bouquet is a candidate bouquet.

    Now let $A$ be a candidate bouquet; as explained above, there is a natural cyclic order on the set 
      $H$ of hairs of $A$. Let $\tilde{H}$ be the order-completion of this cyclic order. The elements of
        $\tilde{H}\setminus H$ are in one-to-one correspondence with the set of homotopy classes of curves to infinity in
        $\C\setminus A$, and, using~\ref{item:bothsides}, the set
        $\tilde{H}$ is easily seen to be order-isomorphic to a circle.

We can now compactify the plane to a space $\tilde{\C}=\C\cup \tilde{H}$ in such a way that
    each hair $h$  of $A$ ends at the corresponding point $h\in H\subset\tilde{C}$. The space $\tilde{\C}$ is homeomorphic
     to the closed unit disc, and under this correspondence the set $A\cup \tilde{H}$ becomes a
     \emph{one-sided hairy circle} in the sense of \cite[Definition~4.7]{aartsoversteegen} (essentially, a circle with 
      a collection of hairs densely attached on one side).

   By \cite[Theorem~4.8]{aartsoversteegen}, any two one-sided hairy circles $X$ and  $Y$ are homeomorphic, and
     any such homeomorphism extends to  a homeomorphism of the sphere.  Observe that any
     homeomorphism must map the base circle of $X$ to that of $Y$.

Applying this result to the above compactifications of two candidate bouquets $A$ and $B$, we see that
   $A$ and  $B$ are ambiently homeomorphic. Furthermore, any homeomorphism that
    preserves the cyclic order of the hairs will extend  to a homeomorphism of their compactifications, and hence
     also to  an orientation-preserving homeomorphism of  the plane.
  \end{proof} 
\begin{remark}
   Alternatively, the proof of Proposition~\ref{prop:lelekfan} (as given in \cite{bulaoversteegen}) can be modified to 
      show that, for any two sets satisfying~\ref{item:lelekfan} 
     and~\ref{item:bothsides}, there is a homeomorphism that preserves the cyclic order of hairs. 
     The first part of the theorem thus follows 
     from the second, which in turn can be proved in the same way as \cite[Theorem~4.1]{aartsoversteegen}. \end{remark}

\section{A topological model for exponential maps} \label{sec:model}

 In~\cite[\S1]{aartsoversteegen}, it was shown that $J(f_a)$ is a Cantor Bouquet for $a\in (-\infty,-1)$. The straight brush constructed there
   was defined in dependence on the parameter $a$, and is hence not easy to analyse directly. Instead, we use a similar construction from
   \cite{topescaping}, which has the advantage that the brush arises in a straightforward way from an abstract dynamical system $\F$, defined without
   referring to a specific exponential map. As a result, the resulting dynamics is easy to analyse, and furthermore can be connected to
   the dynamical plane of the exponential map $f_a$ for \emph{any} $a\in\C$. We now review the necessary
   definitions, and then prove the analog of Theorem~\ref{thm:escapingendpointsexplode} for this model function. 

   For the remainder of the article,
    a sequence $\s=s_0 s_1 s_2 \dots \in \Z^{\N_0}$ of integers shall be called an (infinite) \emph{external address}. Our model lives in the space
   $\R\times \Z^{\N_0}$ (with the product topology). If $x=(t,\s)$, we write 
    $T(x)\defeq t$ and $\extaddr(x)\defeq \s \eqdef s_0(x) s_1(x) \dots$ for the projection maps to the first and second coordinates. 
   We also fix the function
    $F\colon [0,\infty)\to [0,\infty); t\mapsto e^t-1$, which serves as a model for 
    exponential growth\footnote{%
 For many purposes, it would be slightly more convenient to use the function $2F$ instead of $F$ in the definitions. All the basic properties of the 
   construction would remain unchanged under this modification, but $2F$ has the advantage of being uniformly expanding. However, since 
    this does not result in any real simplifications for our purposes, we shall retain the definition as made in~\cite{topescaping}.}. We frequently use elementary properties of the function $F$ that can be verified via high-school calculus; in particular, 
     \begin{align}
        F(x+\delta) &\geq F(x)+F(\delta) \geq F(x) + \delta \qquad\text{and} \label{eqn:Fconvex}  \\
       F^{-n}(x) &\to 0 \qquad (n\to\infty)\label{eqn:Fshrinks}
     \end{align}
    for all $x,\delta\geq 0$. 

  \begin{defn}[A topological model for exponential dynamics \cite{topescaping}]
    Define 
  \[ \F\colon [0,\infty)\times \Z^{\N_0} \to \R\times \Z^{\N_0}; 
         (t,\s) \mapsto (F(t) - 2\pi |s_1| , \sigma(\s)), \]
    where $\sigma$ denotes the shift map on one-sided infinite sequences. We set 
   \begin{align*}
       J(\F) &\defeq \{ x\in [0,\infty)\times \Z^{\N_0}\colon T(\F^n(x)) \geq 0 \text{ for all $n\geq 0$}\} \qquad\text{and} \\
        I(\F) &\defeq \{ x \in J(\F)\colon T(\F^n(x)) \to \infty \}. 
    \end{align*}
    We also denote the one-point compactification of $J(\F)$ by $\hat{J}(\F) \defeq J(\F)\cup\{\infty\}$. 

   Let $\s \in \Z^{\N_0}$. If there is $x\in J(\F)$ with $\extaddr(x)=\s$, then $\s$ is called \emph{exponentially bounded}. 
     We define
      \[ \ts \defeq \begin{cases} \min\{t\geq 0\colon (t,\s)\in J(\F)\} & \text{if $\s$ is exponentially bounded} \\
                       \infty & \text{otherwise}.\end{cases} \]
    If $\s$ is exponentially bounded, then $(\s, \ts)$ is called an \emph{endpoint} of $J(\F)$; if additionally
      $(\s,\ts)\in I(\F)$, then it is called an \emph{escaping endpoint}, and $\s$ is called
    \emph{fast}. We write $E(\F)$ and $\tilde{E}(\F)$ for the sets of endpoints and
      escaping endpoints, respectively. 
 \end{defn}
\begin{remark}[Remark 1]
  In \cite{topescaping}, $J(\F)$ was denoted $\bar{X}$, while $I(\F)$ was denoted $X$. We prefer to use the above notation in analogy to 
    exponential maps. 
  Furthermore,
   external addresses 
   were indexed beginning with index 1 (following the convention established in~\cite{expescaping}). 
   However, it seens more natural to begin with index 0, so, when comparing this paper
    with the relevant results in \cite{topescaping}, 
    all indices should be shifted by $1$. 
\end{remark}

  We refer to \cite[\S3]{topescaping} for further discussion of the motivation behind these definitions. 
    It is proved in \cite[\S9]{topescaping} that $\F|_{J(\F)}$ is topologically conjugate to the map $f_{a}$ on its Julia set, for every $a\in (-\infty,-1)$. 
  In particular, as a consequence of \cite{aartsoversteegen}, we see that $\hat{J}(\F)$ is a Lelek fan. 
  In fact, we can say more if we consider $\hat{J}(F)$ as being embedded in $\R^2$ as follows.
\begin{obs}[Embedding in $\R^2$]\label{obs:embedding}
   It is well-known that $\Z^{\N_0}$ (with respect to lexicographic order) is order-isomorphic to the set $\R\setminus\Q$ of all irrational numbers. 
   Hence there is an order-preserving homeomorphism $\alpha\colon \Z^{\N_0}\to \R\setminus\Q$ (compare \cite[\S1.1]{aartsoversteegen} or 
     \cite[\S3.2]{devaneybouquets}). In the following, we shall hence
    identify $\Z^{\N_0}$ with the set of irrational numbers; thus $J(\F)$ can be thought of as a subset of $\R^2$, via the embedding 
    $(t,\s)\mapsto (t,\alpha(\s))$. For ease of terminology, will usually not distinguish between $J(\F)$ and its image under this embedding. 
\end{obs}

\begin{thm}[The model is a straight brush]
  $J(\F)$ is a straight brush. Hence $\hat{J}(\F)$ is a Lelek fan and $\infty$ is an explosion point for $E(\F)$. 
\end{thm}
\begin{proof}
  This is essentially proved in \cite[Section~1.6]{aartsoversteegen}, although the brush there is defined slightly differently. For this 
    reason, we sketch how to prove the result directly. The fact that $J(\F)$ is closed in $\R^2$, as well as ``hairiness'', 
    follows immediately from the definition of
    $\F$ and $J(\F)$; see \cite[Observation~3.1]{topescaping}. (Observe that, for any $t\geq 0$, the set of addresses with
     $\ts \leq t$ form a compact subset of $\Z^{\N_0}$.)  It is easy to see that $\ts <\infty$ for any bounded sequence $\s$
    (see Lemma~\ref{lem:ts} below), and hence the set of addresses with $\ts < \infty$ is dense. Finally, given
    $x\in J(\F)$, it is easy to construct sequences that converge to $x$ from both above and below; this is also a consequence of
   Theorem~\ref{thm:JQbouquet}~\ref{item:denseinJQ} below (take $Q=0$). 
\end{proof} 
 
 Our main goal for the remainder of this section is to prove the following. 

\begin{thm}[Escaping endpoints of $\F$ explode]\label{thm:Fexplodes}
  The point $\infty$ is an explosion point for $\tilde{E}(\F)$. 
\end{thm}

Theorem~\ref{thm:Fexplodes} implies Theorem~\ref{thm:escapingendpointsexplode} in the case where $a\in (-\infty,-1)$. Moreover, for \emph{every}
   exponential map $f_a$, there is a correspondence between points in $I(\F)$ and the set of escaping points of $f_a$
    (see Theorem~\ref{thm:boettcher} below), 
   and hence the results proved in this section also aid us in establishing Theorem~\ref{thm:escapingendpointsexplode} in general. 

The key idea is to identify suitable subsets of $\hat{J}(\F)$ that are still Lelek fans, but
   which have the property that each of their endpoints belongs to $\tilde{E}(\F)$. 

\begin{defn}[Sub-fans of $\hat{J}(\F)$]
   Let $\s^0\in \Z^{\N_0}$ be arbitrary. We define
     \[ X_{\s^0}(\F) \defeq \{x\in J(\F)\colon |s_n(x)| \geq |s^0_n| \text{ for all $n\geq 0$}\}\]
    and $\hat{X}_{\s^0}(\F) \defeq X_{\s^0}(\F)\cup \{\infty\}$. 
\end{defn}

\begin{thm}[Sub-fans are Lelek fans]\label{thm:subfans}
   If $\s^0$ is exponentially bounded, then $\hat{X}_{\s^0}(\F)$ is a Lelek fan. If, additionally,
    $\s^0$ is fast, then every endpoint of $\hat{X}_{\s^0}(\F)$ belongs to $\tilde{E}(\F)$. 
\end{thm}

 The final claim in the theorem follows immediately from the following simple fact. 
  \begin{obs}[Larger entries mean larger potentials]\label{obs:potentials}
      Let $\s^0,\s\in \Z^{\N_0}$ such that $|s_j|\geq |s_j^0|$ for all $j\geq 1$. Then 
       $t_{\s}\geq t_{\s^0}$. In particular, if $\s^0$ is fast, then so is $\s$.
  \end{obs}
  \begin{proof}
    By definition of $\F$, 
      $T(\F^n(t_{\s}, \s^0)) \geq T(\F^n( t_{\s} , \s)) \geq 0$
     for all $n$. Hence $t_{\s}\geq t_{\s^0}$, as claimed. In particular, if $\s^0$ is fast, then
     $t_{\sigma^n(\s)} \geq t_{\sigma^n(\s^0)}\to \infty$, and $\s$ is fast also. 
  \end{proof}

 In order to prove the first part of Theorem~\ref{thm:subfans}, we shall require some control over the minimal potentials
   $\ts$, as provided by the following fact \cite[Lemma~7.1]{topescaping}, whose proof we include for the reader's convenience.

\begin{lem}[Bounds for minimal potentials]\label{lem:ts}
  Let $\s \in \Z^{\N_0}$, and define
    \begin{equation} \label{eqn:ts*}
      \ts^* \defeq \sup_{n\geq 1} F^{-n}(2\pi |s_n|). 
   \end{equation}
   Then $\ts^* \leq \ts \leq \ts^*+1$. In particular, $\s$ is exponentially bounded 
    if and only if $\ts^*<\infty$, and fast if and only if
      $t_{\sigma^n(\s)}^*\to\infty$. 
\end{lem}
\begin{proof}
  The fact that $\ts^* \leq \ts$ follows immediately from the observation that
     \[ T\bigl(\F\bigl( F^{-1}(2\pi|s_n|),\sigma^{n-1}(\s)\bigr)\bigr)=0 \] for $n\geq 1$, and hence 
     $F^{-n}(2\pi |s_n|) \leq \ts$. 
   On the other hand, it is easy to see that $(\ts^*+1,\s)\in J(\F)$, and hence
     $\ts^*+1 \geq \ts$, using the fact that $F(t+1)\geq 2F(t)+1$ for all $t\geq 0$.
\end{proof}

We also use the following elementary fact about contraction under pullbacks of $\F$. 

\begin{obs}[Backwards shrinking]\label{obs:shrinking}
  Let $x,y\in J(\F)$ and $n\in\N$, and suppose that $|s_j(x)|=|s_j(y)|$ for $j=1,\dots,n$ and that $\delta \defeq T(\F^n(y)) - T(\F^n(x)) > 0$. Then 
      \[ T(x) \leq T(y) \leq T(x) + F^{-n}(\delta). \]
\end{obs}
\begin{proof}
  This is immediate from~\eqref{eqn:Fconvex} and the definition of $\F$. 
\end{proof}

\begin{proof}[Proof of Theorem~\ref{thm:subfans}]
   Since projection to the second coordinate is continuous with respect to the product topology, 
    the set $X_{\s^0}(\F)$ is closed in $J(\F)$. Since $\hat{X}_{\s^0}(\F)$ is clearly connected as a union of intervals all of which have a common
    endpoint,  it follows that $\hat{X}_{\s^0}(\F)$ is a continuum. 
    As $\hat{J}(\F)$ is a Lelek fan, we see that $\hat{X}_{\s^0}(\F)$ is a smooth fan with top $\infty$. (Any nontrivial subcontinuum of a 
    smooth fan is
    either an arc or a smooth fan.) 

  To conclude that $\hat{X}_{\s^0}(\F)$ is a Lelek fan, it remains to show that endpoints are dense in $X_{\s^0}(\F)$. Let 
    $x=(t,\s)\in X_{\s^0}(\F)$, and write $\F^n(x)\eqdef (t_n,\sigma^n(\s))$. We define a sequence $(\s^j)_{j=1}^{\infty}$ of addresses by
     \[ s^j_n \defeq \begin{cases}
                          s_n & \text{if $n\neq j+1$}; \\
                          \left\lceil \frac{F( t_j )}{2\pi} \right\rceil &\text{if $n=j+1$}.\end{cases} \]
     By Definition of $\F$ and $J(\F)$, we have
       \[  F(t_j) = 2\pi |s_{j+1}| + t_{j+1} \geq 2\pi |s_{j+1}|,    \] 
      and hence $|s^j_n|\geq |s_n|\geq |s_n^0|$ for all $j\geq 1$ and $n\geq 0$. So
      $x^j \defeq (t_{\s^j}, \s^j)\in X_{\s^0}(\F)$ for all $j$. Furthermore, by Lemma~\ref{lem:ts} and~\eqref{eqn:Fconvex}, 
	       \begin{align*}     
      t_j \leq  F^{-1}(2\pi |s_{j+1}^{j}|) \leq  t^*_{\sigma^j(\s^j)} & \leq 
                  \max( t_{\sigma^j(\s)}^* , F^{-1}(2\pi|s^j_{j+1}|))  \\ &\leq \max ( t_{\sigma^j(\s)} , t_j +2\pi )
                    = t_j+2\pi.\end{align*}
     By Observation~\ref{obs:shrinking}, Lemma~\ref{lem:ts} and~\eqref{eqn:Fshrinks}, 
      \[ |t - t_{\s^j}| \leq 
        F^{-j}\bigl(| t_j - t_{\sigma^j(\s^j)} |\bigr) \leq F^{-j}(2\pi + 1) \to 0\]
      as $j\to\infty$. Since $\s^j\to\s$ by definition, we see that $x^j \to x$, as desired. 

   As already noted, the final claim of the theorem follows from Observation~\ref{obs:potentials}. 
\end{proof}

\begin{proof}[Proof of Theorem~\ref{thm:Fexplodes}]
  Consider the projection  map  $E(\F)\to \Z^{\N_0}$ to  the second coordinate. This map is 
    continuous, and injective by the definition of  $E(\F)$. 
    Since $\Z^{\N_0}$ is totally separated, it follows  that $E(\F)$ 
    and $\tilde{E}(\F)\subset E(\F)$ are totally separated. 

  On the other hand, for any $\s^0\in \Z^{\N_0}$, let $E( X_{\s^0}(\F))$ denote the set of endpoints of $X_{\s^0}(\F)$. By Theorem~\ref{thm:subfans}, $\hat{X}_{\s^0}(\F)$ is a Lelek  fan, and hence 
     $E(X_{\s^0}(\F))\cup\{\infty\}$ is connected by Proposition~\ref{prop:lelekfan} and the definition of explosion  points.
   Furthermore, by Theorem~\ref{thm:subfans}, $E(X_{\s^0}(\F))\subset \tilde{E}(\F)$ when  $\s$ is  fast, and hence
    \[ \tilde{E}(\F)\cup\infty = \bigcup_{\s^0 \text{ fast}} \left( E(X_{\s^0}(\F)) \cup \{\infty\} \right) \]
     is connected, as desired. 
\end{proof} 

 In order to deduce Theorem~\ref{thm:escapingendpoints}, we shall need a slightly stronger version of Theorem~\ref{thm:Fexplodes}, which applies
   also to the set of points that remain far enough to the right under iteration of $\F$. 

 \begin{thm}[Escaping endpoints in other sub-fans]\label{thm:JQbouquet}
   Let $Q\geq 0$, and define
     \[ J_{\geq Q}(\F) \defeq \{x\in J(\F)\colon T(\F^j(x))\geq Q\text{ for all $j\geq 0$}\}. \]
    Then $J_{\geq Q}(\F)$ is a straight brush (using the embedding from Observation~\ref{obs:embedding}), and hence 
        $J_{\geq Q}(\F)\cup\{\infty\}$ is a Lelek fan. Furthermore:
    \begin{enumerate}[(a)]
       \item If $\s^0\in\Z^{\N_0}$ is such that $(t_{\s^0},\s^0)\in J_{\geq Q}(\F)$, then $X_{\s^0}(\F)\subset J_{\geq Q}(\F)$. \label{item:containssubfans}
       \item The set $\tilde{E}_{\geq Q}(\F) \defeq \tilde{E}(\F)\cap J_{\geq Q}(\F) \subset
                      E(J_{\geq Q}(\F))$ is dense in $J_{\geq Q}(\F)$.
            More precisely, for every $x^0\in J_{\geq Q}(\F)$, there are sequences $(x^{j+})_{j=1}^{\infty}$ and 
            $(x^{j-})_{j=1}^{\infty}$ in $\tilde{E}_{\geq Q}(\F)$ such that $\extaddr(x^{j-}) < \extaddr(x^0) < \extaddr(x^{j+})$ for all $j$
            (with respect to lexicographical order) and such that $x^{j+},x^{j-}\to x^0$.\label{item:denseinJQ}
       \item $\infty$ is an explosion point for $\tilde{E}_{\geq Q}(\F)$.\label{item:JQexplosion}
    \end{enumerate}
 \end{thm}
 \begin{remark}
   For $Q>0$, the Cantor bouquet $J_{\geq Q}(\F)$ will have some endpoints that are escaping points of $\F$, but 
    are not endpoints of  $J(\F)$. (Indeed, $(Q, (0)^{\infty})$ is an example of such a
    point.) In other words, $\tilde{E}_{\geq Q}(\F) \subsetneq E(J_{\geq Q}(\F))\cap I(\F)$.
  \end{remark}
 \begin{proof}
   Note that $J_{\geq Q}(\F)$ is closed in
     $J(\F)$ by definition, proving the ``compact sections'' condition in Definition~\ref{defn:straightbrush}. ``Hairiness'' is also immediate from the definition.
     Furthermore, let $\s\in\Z^{\N_0}$ be exponentially bounded. Then 
     $(t_{\s} + Q, \s)\in J_{\geq Q}(\F)$, and if $(t,\s)\in J_{\geq Q}(\F)$, clearly $(t',\s)\in J_{\geq Q}(\F)$ for all $t' > t$. 
    This implies the first part of the ``density'' condition for straight brushes. Thus it remains to establish
    the second part of this condition (that for any given hair, there are sequences of hairs approximating it from
    above and below). This 
    follows from~\ref{item:denseinJQ}, which we establish below. So $J_{\geq Q}(\F)$ is indeed a straight brush.

   To prove claim~\ref{item:containssubfans}, 
    suppose that $x^0 = (t_{\s^0} , \s^0)\in J_{\geq Q}(\F)$, and let $x\in X_{\s^0}(\F)$. Then we can apply Observation~\ref{obs:potentials} to
    $\s^0$ and $\s\defeq \extaddr(x)$, and all their iterates. Thus
     \[ T(\F^n(x)) \geq t_{\sigma^n(\s)} \geq t_{\sigma^n(\s^0)} = T(\F^n(x_0)) \geq Q \]
    for all $n$, as claimed. 

  Now let us prove~\ref{item:denseinJQ}, which follows similarly as density of endpoints in the proof of Theorem~\ref{thm:subfans}. More 
    precisely, let $x = (t , \s) \in J_{\geq Q}(\F)$; we may assume without loss of generality that
    $x\notin E(\F)$. (Otherwise, apply the result to a sequence of values with the same address as  $x$  and tending to $x$, and diagonalise.) 
   In particular, 
     \begin{equation}\label{eqn:growthofx} t_n \defeq T(\F^n(x)) \geq F^n(\delta) \end{equation}
    by~\eqref{eqn:Fconvex}, where $\delta \defeq t - t_{\s} > 0$. 

   We define addresses $\s^{j+}$ and $\s^{j-}$ by 
     \[ s^{j\pm}_n \defeq \begin{cases}
                          s_n & \text{if $n\leq j$}; \\
                          \pm \left\lceil \frac{F( t_{n-1})}{2\pi} \right\rceil &\text{if $n\geq j+1$}.\end{cases} \]
    Then each address $s^{j\pm}$ is fast by~\ref{eqn:growthofx}. Furthermore, 
     \[ 2\pi |s^{j\pm}_n| \geq  F( t_{n-1} ) > F( t_{\sigma^{n-1}(\s)} ) \geq 2\pi |s_n| \]
     for all $j\geq 1$ and $n\geq j+1$, and hence
     $\s^{j-} < \s < \s^{j+}$. 

    Finally, 
    \[ t_{\sigma^n(\s^{j\pm})} \geq F^{-1}( 2\pi |s^{j\pm}_{n+1}|) \geq t_n \geq Q \]
     for $n\geq j$, and by Observation~\ref{obs:shrinking} also 
    \[ t_{\sigma^n(\s^{j\pm})} \geq t_n \geq Q \]
     for $n=1,\dots , j-1$. Thus $x^{j\pm} \defeq (t_{\s^{j\pm}} , \s^{j\pm})\in \tilde{E}_{\geq Q}(\F)$. The fact that
     $x^{j\pm} \to x^0$ as $j\to \infty$ follows as in Theorem~\ref{thm:subfans}, and the proof of~\ref{item:denseinJQ} is complete.

    Finally, \ref{item:JQexplosion} follows from~\ref{item:containssubfans} and~\ref{item:denseinJQ}
     precisely as in the proof of Theorem~\ref{thm:Fexplodes} above. 
 \end{proof} 
 
\section{Review of exponential dynamics and proof of Theorem \ref{thm:escapingendpoints}}\label{sec:exponentialgeneral}
  We now review some key results concerning the dynamics of exponential maps. The following theorem
    \cite[Theorems~4.2 and~4.3]{topescaping}
    shows that, for any $a\in\C$ and sufficiently large $Q>0$, the set $J_{\geq Q}(\F)$ considered in the previous section captures the essential features of 
    the escaping dynamics of $f_a$.

  \begin{thm}[{Conjugacy between $\F$ and exponential maps}]\label{thm:boettcher} 
    Let $a\in\C$, and consider the exponential map $f_a\colon\C\to\C; f_a(z)=e^z+a$. If $Q$ is sufficiently large, then there exists 
      a closed forward-invariant set $K\subset J(f_a)$ and a homeomorphism $\gconj\colon J_{\geq Q}(\F) \to K$ with the following properties.
   \begin{enumerate}[(a)]
     \item $\gconj$ is a conjugacy between $\F$ and $f$; i.e. 
         $\gconj \circ \F = f_a\circ \gconj$.
     \item  $\re\gconj(x)\to +\infty$ as $T(x)\to\infty$.\label{item:gasymptotic}
     \item The map $\gconj$ preserves vertical ordering. That is, if $\s^1,\s^2\in\Z^{\N_0}$ are exponentially bounded addresses such that
         $\s^1<\s^2$ with respect to lexicographical ordering, then the curve $t\mapsto \gconj(t,\s^1)$ tends to infinity below the curve
         $\gconj(t,\s^2)$.\label{item:verticalorder}
     \item There is a number $R>0$ with the following property: if $z\in\C$ such that $\re f_a^n(z)\geq R$ for all $n\geq 0$, then 
         $z\in K$. 
   \end{enumerate} 
  \end{thm} 
\begin{remark}
  Part~\ref{item:verticalorder} is not stated in \cite[Theorems~4.2 and~4.3]{topescaping}, but follows from the construction, as discussed 
    on p.~1962 of the same paper.
\end{remark}

\begin{cor}[Cantor Bouquets in Julia sets]
    The map $\gconj$ in Theorem~\ref{thm:boettcher} extends to a homeomorphism $\R^2\to \C$ (where we consider
       $J_{\geq Q}(\F) \subset J(\F)$ embedded in $\R^2$ as in Observation~\ref{obs:embedding}). In particular, 
     the set $K=\gconj(J_{\geq Q}(\F))\subset J(f)$ is a Cantor bouquet. 
\end{cor} 
\begin{proof}
  By Theorem~\ref{thm:JQbouquet} and part~\ref{item:verticalorder} of Theorem~\ref{thm:boettcher}, the set $K$ satisfies all the
      requirements of Theorem~\ref{thm:cantorbouquets}, and the claim follows.

 (We remark that, in this case, the compactification discussed in the
    proof of Theorem~\ref{thm:cantorbouquets} is well-known: it is obtained by adding to $\C$ all 
    external addresses at infinity, together with the countable set of ``intermediate external addresses'' 
    (see Definition~\ref{defn:intermediateaddresses} below). This space has the property that 
      $\gconj(t,\s)\to \s$ as $t\to\infty$, for every exponentially bounded address $\s$; compare \cite[Remark on p.357]{nonlanding},
        and also \cite[\S5]{brush} for the same result in a more general setting. 
      The compactification is known to be homeomorphic to the closed unit disc, and 
      hence we can obtain the result directly from~\cite[Theorem~4.8]{aartsoversteegen} using this compactification,
      without requiring the  more general Theorem~\ref{thm:cantorbouquets}.)
\end{proof}

In order to complete the proof of Theorem~\ref{thm:escapingendpoints}, we shall require two further ingredients. The first concerns pullbacks of
  unbounded sets under entire functions.

\begin{lem}[Pullbacks]\label{lem:pullbacks}
  Let $f\colon\C\to\C$ be any non-constant entire function, and let $A\subset\C$ be a set such that $A\cup\{\infty\}$ is connected. Then
    $f^{-1}(A)\cup\{\infty\}$ is connected.
\end{lem}
\begin{proof}
 We prove the contrapositive. Suppose that $A\subset\C$ is such that $f^{-1}(A)\cup\{\infty\}$ is disconnected.  Then there is a nonempty 
    bounded relatively open and closed
    subset $X\subset f^{-1}(A)$. Since $f$ is an open mapping, $f(X)$ is open in $A$. Furthermore, the closure $\overline{X}$ (in $\C$) is compact, 
    and hence $f(X)$ is bounded and relatively closed in $A$ by continuity of $f$. Thus $A\cup\{\infty\}$ is disconnected, as required.
\end{proof} 

The second ingredient concerns escaping endpoints: To prove Theorem~\ref{thm:escapingendpoints}, we need to know that
   the image of an escaping endpoint for $\F$ under the conjugacy $\gconj$ is also an escaping endpoint for the exponential map $f_a$. 
   With the usual definition of rays and endpoints for $f_a$, as introduced in~\cite{expescaping}, this would be immediate. 
   For our Definition~\ref{defn:endpoints}~-- which is easier to state, but less natural~-- this is less obvious, but it was shown in~\cite{frs} that 
   the two notions coincide (compare \cite[Remark~4.4]{frs}). In particular, the following is shown there.

\begin{thm}[Path-connected components of the escaping set]\label{thm:frs}
   Let $a\in\C$ be arbitrary, and suppose that $\gamma\colon [0,1]\to I(f_a)$ is a continuous curve in the escaping set $I(f_a)$. Then for all $R>0$ there
    exists $n\geq 0$ such that $\re( f^n(\gamma(t))) \geq R$ for all $t\in [0,1]$. 
\end{thm}
\begin{proof}
 While this is not stated explicitly in this form in \cite{frs}, it is a direct consequence of 
   Theorem~4.2 given there, applied to the setting of Corollary~4.3 in the same paper.
\end{proof} 

\begin{cor}[Characterisation of rays and endpoints]\label{cor:endpointcharacterisation}
  Let $a\in\C$ be arbitrary, and let $\gconj$ and $K$ be as in Theorem~\ref{thm:boettcher}. Let $z\in I(f_a)$, and let $n\geq 0$ be sufficiently large
    such that $f_a^n(z)\in K$. Then either
   \begin{enumerate}[(a)]
     \item $\gconj^{-1}(f_a^n(z))\in \tilde{E}(\F)$, and $z$ is an escaping endpoint of $f_a$ in the sense of Definition~\ref{defn:endpoints}; or
     \item $\gconj^{-1}(f_a^n(z))\in I(\F)\setminus\tilde{E}(\F)$, and $z$ is on a hair in the sense of Definition~\ref{defn:endpoints}. 
       \label{item:rays}
    \end{enumerate}
\end{cor}
\begin{proof}
  First observe that, by property~\ref{item:gasymptotic} of Theorem~\ref{thm:boettcher}, the map $\gconj$ maps (non-)esca\-ping points of $f_a$ to (non-)escaping
     points of $\F$.  Hence either $x \defeq \gconj^{-1}(f_a^n(z)) \in \tilde{E}(\F)$ or $x\in I(\F)\setminus\tilde{E}(\F)$. 
     Observe also that, since the set of escaping endpoints is completely invariant under $\F$, which of the two alternatives holds is independent
     of the choice of $n$. 

   If $x$ is a non-endpoint for $J(\F)$, then~-- by increasing the number $n$ if necessary~--
   we may assume that 
     $x$ is also a non-endpoint for $J_{\geq Q}(\F)$, where $Q$ is as in Theorem~\ref{thm:boettcher}. 
    It is then immediate that $f_a^n(z)$ is on a hair, and by applying a local inverse of $f_a^{-n}$, we obtain an arc in $I(f)$ containing
    $z$ as an interior point, as required. 

  If $x\in \tilde{E}(\F)$, then we see analogously that $z$ is accessible from the escaping set $I(f_a)$. Hence it only remains to prove that, if
    $\gamma\subset I(f_a)$ is an arc, then $\gamma$ cannot contain $x$ in its interior. This follows from
     Theorem~\ref{thm:frs} (again, increasing $n$ if necessary). 
\end{proof} 

\begin{proof}[Proof of Theorem~\ref{thm:escapingendpoints}]
  Let $\gconj$, $K$ and $Q$ be as in Theorem~\ref{thm:boettcher}, and define 
      \[ A\defeq \gconj(\tilde{E}(\F)\cap J_{\geq Q}(\F)) \subset \tilde{E}(f_a) \]
    (where the last inclusion follows from Corollary~\ref{cor:endpointcharacterisation}). 
    Then $A$ is forward-invariant under $f_a$, and  $\infty$ is an explosion point for $A\cup\{\infty\}$ by 
    Theorem~\ref{thm:JQbouquet}. The increasing union 
      \[ \bigcup_{n\geq 0} ( f^{-n}(A) \cup\{\infty\}) \]
     is dense in $\tilde{E}(f_a)\cup\{\infty\}$, since the  backward orbit of any point other than the omitted value
      $a$ is dense in the  Julia set (by Montel's theorem). It is also
     connected by Lemma~\ref{lem:pullbacks}; 
     the theorem follows.
\end{proof}

\begin{rmk}[Fast escaping points]\label{rmk:fast}
  Let $\s\in \Z^{\N_0}$ be an exponentially bounded address such that the entries of $\s$ grow in an iterated exponential fashion;
   e.g.\ $s_0=m$, for some $m\in\N$, and $s_{k+1} \defeq 2^{s_k}$. If $m$ is large enough, then by Theorem~\ref{thm:JQbouquet} 
    we can replace
    $J_{\geq Q}(\F)$ by $X_{\s}(\F)\subset J_{\geq Q}(\F)$ in the proof of Theorem~\ref{thm:escapingendpoints}. Furthermore, in this case all points of 
    $\gconj( X_{\s}(\F))$ belong to the fast escaping set $A(f_a)$ discussed after the statement of Theorem~\ref{thm:escapingendpoints}. 
   This shows that the set of fast escaping endpoints, together with infinity, is also connected.
\end{rmk}

\section{Dense path-connected components}\label{sec:denseray}
  We now prove Theorem~\ref{thm:denseray}, in the following more general form.


\begin{thm}[Accumulation points]\label{thm:accumulation}
  Let $a\in\C$. Let $C$ be a path-connected component of $I(f_a)$, and suppose that 
     $z_0,z_1\in \tilde{E}(f_a) \cup (\C\setminus I(f))$ both belong to the closure $\overline{C}$. 
     Then $z_0$ and $z_1$ are not separated in $\{z_0,z_1\}\cup \tilde{E}(f_a)$. 
\end{thm}

The idea of the proof is as follows. If $z_1$ and $z_2$ were separated, then by Lemma~\ref{lem:planeseparation}, 
    they could be separated by a closed subset of the plane,
   which by definition does not contain any escaping endpoints. However, this set would have to intersect $C$, which itself essentially sits
   within a Cantor bouquet (obtained as preimages of the set $K$ from Theorem~\ref{thm:boettcher}). We thus obtain a contradiction to 
  Theorem~\ref{thm:Fexplodes}. To make this precise, we notice the following simple fact.

\begin{prop}[Accessing Cantor Bouquets]\label{prop:accessing}
  Let $X\subset\C$ be a Cantor bouquet, and let $E\subset X$ be a dense subset of $X$ such that 
   $E\cup\{\infty\}$ is connected. 

  Suppose that $\Delta\not\subset X$ is a closed connected subset of $\C$ with $\Delta\cap E=\emptyset$. 
   Then every connected component $I$ of $\Delta\cap X$ is either unbounded or contains an endpoint of $X$. 
\end{prop}
\begin{proof}
  Clearly it is enough to prove this in the case where $X\subset \R^2$ is a straight brush. In this case,
    let us prove the contrapositive. Suppose that $I$ is bounded and does not contain an endpoint of $X$; we will show that $E\cup\{\infty\}$ is disconnected.
     Then
     $I=[\tau_1,\tau_2]\times \{\alpha\}$ for some $\alpha\in\R\setminus\Q$ and
      $t_{\alpha} < \tau_1 \leq \tau_2$. Choose $T_1\in (t_\alpha, \tau_1)$
       and $T_2 >\tau_2$ such that $(T_1,\alpha) , (T_2,\alpha) \notin \Delta$.

    Consider the segment $J\defeq [T_1,T_2]\times \{\alpha\}$, and let $\delta>0$ be sufficiently small to ensure that 
      $(T_j, \beta)\notin \Delta$ whenever $j=1,2$ and $|\alpha-\beta|\leq \delta$. 
      Let $\beta_n^+\searrow \alpha$ and $\beta_n^-\nearrow \alpha$ be sequences of rational numbers converging to
        $\alpha$ from above and below such that $|\beta_n^{\sigma}-\alpha|\leq \delta$ for all $n$ and $\sigma\in\{+,-\}$. 
       Then the line segments $J_n^{\pm} \defeq [T_1,T_2] \times \{\beta_n^{\pm}\}$ converge to $J$ from above and below
       and are disjoint from $X$. 
       Consider the open rectangles
       $R_n\subset \R^2$ bounded by $J_n^+$ and $J_n^-$ together with the vertical line segments connecting their endpoints. 

   Fix some point $y\in \Delta\setminus X$; then $\partial R_{n_1}$ separates $y$ from
     $I$ for sufficiently large $n_1$. Let $K_1$ be the closure of the connected component of $K\cap R_{n_1}$ that contains $I$.
     Then, by the boundary bumping theorem~\cite[Theorem~5.6]{continuumtheory}, there is $\sigma\in \{+,-\}$ such that $K_1$ intersects
     $J_{n_1}^{\sigma}$. Now let $n_2> n_1$, and consider the ``sub-rectangle'' $R^{\sigma}_{n_1,n_2}$ of $R_n$ whose horizontal sides are $J_{n_1}^{\sigma}$ and $J_{n_2}^{\sigma}$. 
     By another application of the boundary bumping theorem, there is a compact connected subset $K_2\subset K_1\cap \overline{R_{n_1,n_2}^{\sigma}}$ 
    that connects
     these two horizontal sides. 

   If $n_2$ is sufficiently large, then we can pick points $a_1,a_2\in X\cap R_{n_1,n_2}^{\sigma}$ that are near the two 
      vertical sides; by density of $E$ we can furthermore assume that $a_1,a_2\in E$. Then $K_2$ separates $a_1$ and $a_2$ in
    $R^{\sigma}$.

   Now let $M < 0$, and consider the set  $\Gamma$ obtained as the union of 
       $K_2$, $[M,T_2]\times \{\beta_{n_1}^{\sigma}\}$, $[M,T_2]\times \{\beta_{n_2}^{\sigma}\}$ and
      the vertical interval connecting the left endpoints of the latter two segments. The two horizontal intervals in $\Gamma$ and its left vertical side
      are disjoint from $X$, while $K_2\subset K$ is disjoint from $E$ by assumption. So 
      $\Gamma\cap E = \emptyset$. Moreover, clearly $\Gamma$ also separates $a_1$ from $a_2$ in $\C$ (this follows, for example,
     from Janiszewski's theorem \cite[p.~31]{pommerenkeunivalent}). 
     This proves that $E$ is disconnected, as required.
\end{proof}

\begin{proof}[Proof of Theorem~\ref{thm:accumulation}]
 Let us suppose, by contradiction, that $z_0$ and $z_1$ are separated in $\{z_0,z_1\}\cup \tilde{E}(f_a)$. Then there exists 
    a closed connected set $\Delta\subset \C$ such that $z_0$ and $z_1$ belong to different complementary components of $\C\setminus \Delta$
     and $\Delta\cap \tilde{E}(f_a)=\emptyset$. 
    Let $\delta$ be sufficiently small that $\overline{B(z_j,\delta)}\cap \Delta=\emptyset$ for $j=0,1$. By assumption, there exists a continuous curve
     $\gamma\colon [0,1]\to I(f_a)$ with $|\gamma(0) - z_0| < \delta$ and $|\gamma(1)-z_1| < \delta$. It follows that $\gamma([0,1])\cap \Delta\neq \emptyset$; let
    $\zeta_0$ be a member of this intersection.

 By Theorem~\ref{thm:frs}, there exists a number $n$ such that $\gamma_n\defeq f^n(\gamma([0,1]))\subset K$, where $K$ is the Cantor Bouquet from
    Theorem~\ref{thm:boettcher}. In particular, $\gamma_n$ is an arc. If $V_n$ is a sufficiently small neighbourhood of
    $\gamma_n$, then we can find a branch $\phi_n$ of $f_a^{-n}$ on $V_n$ that takes $\gamma_n$ to $\gamma$ and extends
    continuously to $\partial V_n$. Set $V\defeq \phi_n(V_n)$.  

 Let $\Delta_0$ be the connected component of $\Delta\cap \overline{V}$ containing $\zeta_0$; by the boundary bumping theorem, 
    we have $\Delta_0\cap \partial V\neq\emptyset$. Now we can apply Proposition~\ref{prop:accessing} to
    the Cantor Bouquet $K$, the set $E \defeq \tilde{E}(f)\cap K$ (which has the desired properties by Theorem~\ref{thm:JQbouquet}), and 
    the continuum $\Delta_n \defeq f^n(\Delta_0)$. Note that $\Delta_n\not\subset K$ since $\Delta_1\not\subset\gamma_n$. 

 By construction, $\Delta_n\cap E = \emptyset$, but the connected component of
    $\Delta_n\cap X$ containing $f^n(z_0)$ belongs to the interior of the arc $\gamma_n$, and hence is bounded and contains no endpoint of $K$.
    This is a contradiction.
\end{proof}

\section{Maps with accessible singular values}\label{sec:accessible}

  The key point about the hypotheses in Theorem~\ref{thm:escapingendpointsexplode} is that the singular value $a$ can be connected to
    $\infty$, by a curve $\gamma$ that either belongs to the Fatou set (if $a\in F(f_a)$), or is a piece of a ``dynamic ray'' in the escaping set (see below). 
    The preimage of $\gamma$, which contains no endpoints,  
     then cuts the complex plane into countably many strips, and we can study symbolic dynamics by considering the 
    ``itineraries'' of different points with respect to this partition. 

 \begin{defn}[Itineraries]\label{defn:Gammaitineraries}
    Let $f\colon \C\to\C$ be a continuous function, and let $\Gamma\subset\C$ be a closed set. 
       We say that $z\in\C$ \emph{has an itinerary} with respect to $\Gamma$ if $f^n(z)\notin \Gamma$ for all $n\geq 0$. 
        In this case, the itinerary of $z$ is defined to be the sequence of connected components of $\C\setminus\Gamma$ visited by $f^n(z)$, for $n\geq 0$.

    In particular, two such points $z$ and $w$ have the same itinerary if $f^n(z)$ 
        and $f^n(w)$ belong to a common connected component of $\C\setminus\Gamma$ for all $n$; otherwise,
        they have different itineraries. 
 \end{defn}

When two points have different itineraries, they can clearly be separated within the 
    set of endpoints:

  \begin{obs}[Separation from itineraries]\label{obs:itineraryseparation}
    Let $f\colon \C\to\C$ be a continuous function, and let $\Gamma\subset\C$ be a closed set. 
     Suppose that $z,w\in\C$ have different itineraries with respect to $\Gamma$.
    Then $z$ and $w$ are separated in 
       $\displaystyle{\C\setminus \bigcup_{n\geq 0}f^{-n}(\Gamma)}$.
  \end{obs}
\begin{proof}
    Let $n_0\geq 0$ be such that $f^{n_0}(z)$ and $f^{n_0}(w)$ belong to different connected components of $\C\setminus\Gamma$. Since
    $\Gamma$ is closed, the two points are also separated in $\C\setminus \Gamma$. Then
    $z$ and $w$ are separated in $\C\setminus f^{-n_0}(\Gamma)$
   by Observation~\ref{obs:preimageseparation}.
 \end{proof}

 \subsection*{Dynamic rays of exponential maps}
   A full description of the escaping set of an arbitrary exponential map in terms of
    \emph{dynamic rays} was first given in \cite{expescaping}. 
    We shall now review their  definition and basic properties. 

   \begin{defn}[Dynamic rays]
     Let $f_a$ be an exponential map. A \emph{dynamic ray} of $f_a$ is a maximal injective continuous curve
       $g\colon (0,\infty)\to I(f_a)$ such that 
        \begin{enumerate}
             \item $\lim_{t\to\infty} \re f_a^n(g(t)) = \infty$
      uniformly in $n$, and
              \item for all $t_0>0$, $\lim_{n\to\infty} \re f_a^n(g(t)) = \infty$ uniformly for $t\geq t_0$. 
     \end{enumerate}
     If additionally $z_0 \defeq \lim_{t\to 0}g(t)$ is defined, then we say that $g$ \emph{lands} at $z_0$. 
   \end{defn}

 We shall use the following basic properties of dynamic rays. 

 \begin{thm}[Properties of dynamic rays]\label{thm:rays}
    Let $f_a$ be an exponential map, and let $\gconj$ be the map from Theorem~\ref{thm:boettcher}. 
      \begin{enumerate}[(a)]
        \item For every exponentially bounded external address $\s\in \Z^{\N_0}$, there is a unique (up to reparameterisation) dynamic
           ray $g_{\s}$ such that $\extaddr(\gconj^{-1}(g_{\s}(t)))=\s$ for all sufficiently large $t$. Conversely, if $g$ is a dynamic ray of $f_a$, then
           $g=g_{\s}$ for some $\s$ (again up to reparameterisation).
         \item The vertical order of dynamic rays (for $t\to\infty$) coincides with the lexicographical ordering of their external addresses.
       \item If $z\in I(f)$ is on a hair, then either $z$ is on a unique dynamic ray, or there is $n_0\geq 1$ such that
                  $f^n(z)$ is on a unique dynamic ray containing the singular value $a$.
      \item If $z$ is an endpoint of $f_a$, then either $z$ is the landing point of a dynamic ray (which is unique if additionally
                   $z\in I(f)$), or there is $n_0\geq 1$ such that 
            $f^{n_0}(z)$ is the landing point of the unique dynamic ray containing the singular value $a$.\label{item:thm:rays:endpoints}
      \end{enumerate}
 \end{thm}
 \begin{proof}
  These are well-known properties of dynamic rays (compare~\cite{expescaping} or \cite{topescaping}), and follow easily from
     Theorem~\ref{thm:boettcher} and Corollary~\ref{cor:endpointcharacterisation}.
 \end{proof}

 \subsection*{Curves at the singular value}
  We now study exponential functions satisfying the assumptions of Theorem~\ref{thm:escapingendpointsexplode}.
    For every such map $f_a$, we can find a natural curve connecting the singular value to $\infty$,
    whose preimages give rise to a natural dynamical partition of the Julia set. The study of these partitions, and of the resulting
     itineraries, is well-developed. Here we shall present the basic facts we require and refer e.g.\ to \cite{expper,expcombinatorics,nonlanding}
     for further background. 
 
   Let us first consider the case where 
     $a\notin F(f_a)$. Then Theorem~\ref{thm:rays} immediately implies the following.
 \begin{cor}[Rays landing at singular values]\label{cor:accessible}
    Suppose that $f_a$ is an exponential map for which $a$ is either an endpoint or on a hair. Then there is 
      an exponentially bounded address $\s$ such that either $g_{\s}$ lands at $a$, or such that $a\in g_{\s}$. 
 \end{cor}

  Now consider the case where the singular value belongs to the Fatou set. The following is well-known, and follows from
   the fact that $a$ is the unique singular value of $f_a$, and that exponential maps have no wandering domains.
  \begin{obs}[Singular values in the Fatou set]
     Let $f_a$ be an exponential map. Then $a\in F(f_a)$ if and only if $f_a$ has an attracting or parabolic cycle. In this case, there is a unique cycle
      of periodic components of the Fatou set, and $a$ belongs to one of the components of this cycle.
  \end{obs}

   If the component $U$ of $F(f_a)$ containing the singular value has period $n$, 
     we can connect $a$ and $f_a^n(a)$ by an arc $\gamma_0$ in $U$. Let $\gamma$ be the component of
     $f_a^{-n}(\gamma_0)$ containing $a$; then $\gamma$ is a curve connecting $a$ to $\infty$ in $U$. In order
     to associate symbolic dynamics to the curve $\gamma$, we should ``fill in the gaps'' within 
     $\Z^{\N_0}$ with finite sequences called ``intermediate external addresses''; compare
       \cite[\S3]{expattracting} or 
      \cite[\S2]{expcombinatorics}. 

  \begin{defn}[Intermediate external addresses]\label{defn:intermediateaddresses}
     An \emph{intermediate external address} (of length $n\geq 1$) is a finite sequence 
              \[ \s = s_0 s_1 s_2 \dots s_{n-2} \infty, \]
     where $s_j\in \Z$ for $j<n-2$ and $s_{n-2}\in \Z+1/2$. The union of $\Z^{\N_0}$ (the space of infinite external addresses) and the
     set of all intermediate external addresses is denoted $\Sequb$; we also write $\Sequ \defeq \Sequb\setminus \{\infty\}$. 

    Let $f_a$ be an exponential map and let $\gamma\subset \C\setminus I(f_a)$ be a curve connecting some finite endpoint to infinity. Then we say that 
     $\extaddr(\gamma)=\s\in\Sequb$ if, for any exponentially bounded addresses $\addR^1$, $\addR^2$, the three curves 
      $g_{\addR^1}\bigl([1,\infty)\bigr)$, $\gamma$ and $g_{\addR^2}\bigl([1,\infty)\bigr)$ are ordered in positive orientation with respect to cyclic order at infinity if and only if 
       $\addR^1 < \s < \addR^2$ with respect to the cyclic order on $\Sequb$. 
  \end{defn}
\begin{remark}[Remark 1]
 At first, the notion of a curve having address $\s$ may seem to depend on the parameterisation of
       dynamic rays. However, this is not the case, since using different ``tails'' of the same ray will result in the same cyclic order; the interval
       $[1,\infty)$ is used merely for convenient notation.
\end{remark}
 \begin{remark}[Remark 2]
  Since the cyclic order on $\Sequb$ is order-complete (in fact, the space is isomorphic to the circle), 
     $\extaddr(\gamma)$ is defined for every curve $\gamma$ as in
    Definition~\ref{defn:intermediateaddresses}. Observe that, if $\extaddr(\gamma)\neq\infty$, then 
     $\re z\to +\infty$ along $\gamma$. In particular, $f(\gamma)$ is then also a curve to infinity, and $\extaddr(f_a(\gamma)) = \sigma(\extaddr(\gamma))$. 
 \end{remark}
  
  Now let us return to our case of an exponential map $f_a$ having an attracting or parabolic orbit, and the curve $\gamma$ defined above. 
    Clearly we have      $\extaddr(f_a^{n-1}(\gamma)) = \infty$, and hence 
    $\extaddr(\gamma)$ is an intermediate external address of length $n$.  

\subsection*{Itineraries and escaping endpoints}
  For the remainder of the section, we fix an exponential map $f_a$ satisfying the hypotheses of Theorem~\ref{thm:escapingendpointsexplode}, the
    curve $\gamma$ connecting $a$ to $\infty$ constructed above, and the associated address $\s\in \Sequ$. That is, if 
    $a$ is an endpoint or on a hair, then $\gamma$ is the piece of  $g_{\s}$ beginning at $a$, where $g_{\s}$ is the ray from Corollary~\ref{cor:accessible}.
      If 
    $a\in F(f_a)$, then $\gamma$ and $\s=\extaddr(\gamma)$ are as defined above. In particular, in either case $\gamma$ contains no endpoints,
    except possibly $a$ itself, and $\s$ is either intermediate or exponentially bounded. 

 \begin{remark}
   In the case where $a$ belongs to an invariant component of the Fatou set,
    the construction above leads to a curve $\gamma$ having real parts tending to $-\infty$, and hence having address $\s=\infty$. This choice, 
    while correct, would require us to introduce some tedious notation for special cases below. Instead, we simply note that we can always replace
     $\gamma$ by a piece of its preimage, extended to connect to the singular value; this curve will then have 
     an address in $\Sequ$. (We remark that the case where $\s=\infty$ is trivial, anyway, since here no two rays can share the same itinerary; in fact, 
      it is well-known that the Julia set is then a Cantor bouquet.)  In the following, we will hence always suppose that $\s\neq\infty$. 
 \end{remark}

 Now $f_a^{-1}(\gamma)$ consists of countably many curves from $-\infty$ to $+\infty$, which cut the plane into countably many strips,
    and we can study itineraries with respect to $\Gamma = f^{-1}(\gamma)$ (in the sense of Definition~\ref{defn:Gammaitineraries})
     in purely combinatorial terms. 

 \begin{defn}[Combinatorial itinerary]
   Let  $\s\in\Sequ$, and let 
    $\raddr\in\Z^{\N_0}$ be an address with $\sigma^n(\raddr)\neq \s$ for all $n\geq 1$. We define 
      the \emph{(combinatorial) itinerary} 
     $\itin_{\s}(\raddr)= \addm\in\Z^{\N_0}$ by the condition
      \begin{equation}\label{eqn:kneading} m_j\s < \sigma^j(\raddr) < (m_j+1)\s \end{equation}
     for 
      $j\geq 0$. 

    We also define the \emph{kneading sequence} $\K(\s) \defeq \itin_{\s}(\s)$. Observe that, in the case where 
      $\s$ is an intermediate external address of length $n$, or periodic of period $n$, only the first $n-1$ entries $u_0,\dots,u_{n-2}$ 
      of the kneading sequence
     can be defined according to~\eqref{eqn:kneading}; in this case we set $\K(\s) \defeq u_0 \dots u_{n-2} *$. 
 \end{defn}
\begin{remark}
  Compare \cite[\S3]{expcombinatorics} for a further discussion of combinatorial
     itineraries and kneading sequences. There, itineraries are also defined for intermediate external addresses, and iterated preimages of the address $\s$.
     In particular, for simplicity, our definition of the kneading sequence for periodic $\s$ is slightly different from~--
     and less accurate than~-- 
     that given in \cite{expcombinatorics}; in particular, for our purposes periodic addresses do not have 
     periodic kneading sequences.
\end{remark}

     Clearly, by Observation~\ref{obs:itineraryseparation}, 
     if two exponentially bounded addresses as above have different combinatorial itineraries, 
     then the corresponding two rays and their landing points are separated by $\bigcup_{n\geq 1}f_a^{-n}(\gamma)\subset I(f_a)\setminus E(f_a)$. 
    Hence we can prove results about the separation of endpoints by studying the properties
    of sets of addresses sharing the same itinerary. 
    The following simple observation shows that the itineraries of such 
    addresses must satisfy certain restrictions
    in terms of the kneading sequence. 

  \begin{lem}[Addresses sharing an itinerary~{\cite[Lemma~2.3]{nonlanding}}]\label{lem:sharing}
    Let $\s\in\Sequ$, and suppose that $\addR^1\neq \addR^2$ are two addresses, neither an iterated preimage of $\s$, with
      $\addm \defeq \itin_{\s}(\addR^1) = \itin_{\s}(\addR^2)$. Then $\sigma^k(\addR^1)\neq \sigma^k(\addR^2)$ for all $k\geq 0$.

    Furthermore, let $j\geq 1$ with $r_{j-1}^1 \neq r_{j-2}^2$. Then, for all $k\geq 0$, $m_{j+k} = u_{\ell}$ for some $\ell\leq k$, where $\addu := \K(\s)$.
  \end{lem}
  \begin{remark}
    In the case where $\s$ is intermediate or periodic, the final statement implies, in particular, that $m_{j+k}$ agrees with one of the finitely many
     integer entries of $\addu$.  
   \end{remark}

 The following will allow us to complete the proof of Theorem~\ref{thm:escapingendpointsexplode}. 
  
   \begin{prop}[Addresses sharing an itinerary are slow]\label{prop:slowsharingitineraries}
      Let $\s\in\Sequ$ be either intermediate or exponentially bounded. 
       If $\addR^1\neq\addR^2$ are such that the itineraries $\itin_{\s}(\addR^1)$ and $\itin_{\s}(\addR^2)$ are defined and coincide, then
       $\addR^1$ and $\addR^2$ are slow.
   \end{prop}
   \begin{proof}
       By the first part of Lemma~\ref{lem:sharing}, 
       the addresses $\addR^1$ and $\addR^2$ differ at infinitely many indices. 
       Furthermore, if $\s$ is exponentially bounded, we have $t \defeq \ts < \infty$, and by Lemma~\ref{lem:ts}, 
        $2\pi|s_n|\leq F^n(\ts^*) \leq F^n(\ts)$ for all $n\geq 0$. On the other hand, if $\s$ is intermediate of length $K$, 
       then it contains only finitely many 
        entries, and we can choose $t$ such that $2\pi |s_n| \leq F^n(t)$ for $n\leq K-2$. 

      Let $N\geq 1$ be such that $r_{N-1}^1 \neq r_{N-1}^2$. Then, by Lemma~\ref{lem:sharing}, we have
           \[ 2\pi |r^{j}_n|  \leq 2\pi( \max_{k\leq n-N} (|s_k|+ 1 ) \leq F^{n-N}(t+2\pi) \]
        for all $n \geq N$ and $j=1,2$. Hence 
            $t^*_{\sigma^{N}(\addR^{j})} \leq t+2\pi$.
        Since $N$ can be chosen arbitrarily large, we see that 
         $\addR^1$ and $\addR^2$ are slow, as claimed. 
   \end{proof}

\begin{proof}[Proof of Theorem~\ref{thm:escapingendpointsexplode}]
    By Theorem~\ref{thm:escapingendpoints}, it remains to show that $\tilde{E}(f_a)$ is totally separated for all $a\in\C$ satisfying the hypotheses
     of the theorem.  

    So let $z,w\in \tilde{E}(f_a)$ with $z\neq w$. By Observation~\ref{obs:itineraryseparation}, it is enough to show that
     $z_n\defeq f_a^n(z)$ and $w_n\defeq f_a^n(w)$ are separated by $f_a^{-1}(\gamma)$ for some $n\geq 0$. (Here $\gamma$ is the curve connecting the
     singular value to $\infty$, as above.) 

    If $z_{n+1}=w_{n+1}$ for some (minimal) $n\geq 0$, then 
      $z_n$ and $w_n$ differ by a non-zero integer multiple of $2\pi i$, and are hence separated by $f^{-1}(\gamma)$, as desired. 

    So suppose that $z_n\neq w_n$ for all $n$. 
     By Theorem~\ref{thm:rays}~\ref{item:thm:rays:endpoints}, there is $n\geq 0$ such that
     $z_n$ and $w_n$ 
      are landing points of two dynamic rays $g_{\addR^z}$ and $g_{\addR^w}$, at fast external addresses. If $n$ is large enough, then neither
     $\addR^z$ nor $\addR^w$ is mapped to the address $\s$ of $\gamma$ under itation of the shift map. 
 
     By Proposition~\ref{prop:slowsharingitineraries}, we have $\itin_{\s}(\addR^z)\neq \itin_{\s}({\addR^w})$. If $j$ is the first index
      at which the two itineraries differ, then $g_{\sigma^j(\addR^z)}$ and $g_{\sigma^j(\addR^w)}$, and hence $z_{n+j}$ and $w_{n+j}$,
      are separated by $f^{-1}(\gamma)$, and the proof is complete.
\end{proof}

\subsection*{Itineraries and arbitrary endpoints}
 In order to also prove Theorem~\ref{thm:endpoints} in the case where $a\notin F(f_a)$, 
     we shall use the following combinatorial statement, which concerns
    addresses sharing an itinerary in general. 

  \begin{prop}[Number of addresses sharing a common itinerary]\label{prop:finitenumber}
    Let $\s\in \Sequ$ and let $\addm\in\Z^{\N_0}$. Let $\mathcal{R}$ denote the set of external addresses whose itinerary with respect to
     $\s$ is defined and agrees with 
     $\addm$. 

     If $\addm$ is periodic or pre-periodic, and $\sigma^j(\addm)\neq \K(\s)$ for all 
        sufficiently large $j\geq 0$, then $\mathcal{R}$ is a finite set of periodic or pre-periodic addresses
       (all having the same period and pre-period). 
      If $\addm$ is not eventually periodic, then $\# \mathcal{R} \leq 2$. 
  \end{prop} 
\begin{remark}[Remark 1]
 The hypothesis of the proposition is necessary. Indeed, suppose that 
     $\K(\s)$ is periodic; recall that for us this implies that $\s$ itself  is not a periodic address.
      If $\sigma^j(\addm) = \K(\s)$ for some $j\geq 0$, then 
        the set $\mathcal{R}$ consists of uncountably many addresses.
     (The reader may think, by way of illustration, of the case of an exponential map with a Siegel disc. Although
      it is not known whether there are such maps for which the singular value is the landing point of a ray, this is
       conjectured to be true  at least for certain rotation numbers; compare \cite{nonlanding}. In this case, we  may
       think of $\mathcal{R}$ as representing the addresses of those dynamic rays that accumulate on the 
       boundary  of  the  Siegel disc.)
\end{remark}
\begin{remark}[Remark 2]
  In the periodic case, more can be said than is stated in the proposition.  For example, the
     addresses in $\mathcal{R}$ belong to at most two periodic cycles;  if
     $\# \mathcal{R}>2$, then they belong to a single cycle; see~\cite[Lemma~5.2]{expattracting}.
\end{remark}
 \begin{proof}
   First suppose that $\addm$ is eventually periodic. We  may assume  that $\addm$ is periodic;
      otherwise,  we apply the proposition to a periodic iterate  of  $\addm$ under the shift,  and
      obtain the result for $\addm$ by pulling back corresponding to the finitely many first entries of $\addm$.

    In  the periodic case, the claim follows directly from \cite[Lemma~3.8]{expcombinatorics}, with one exception.
       This exception concerns the case where $\s$ is periodic, say of period $n$, with kneading sequence
     $\K(\s) = u_0 u_1 \dots u_{n-2} *$, and where 
     $\sigma^j(\addm) = (u_0 u_1 \dots u_{n-2} m_{n+j-1})^{\infty}$ for some $j\geq 0$, where 
       $m_{n+j-1} \in \{ s_{n-1} , s_{n-1} -1\}$.  In these circumstances, the hypotheses of our proposition are
       satisfied, but \cite[Lemma~3.8]{expcombinatorics} does not apply.
 
    We claim that, under the above assumptions, $\mathcal{R}$ consists of periodic or pre-periodic
       addresses, of period $n$. (However, it may be that $\mathcal{R}=\emptyset$, which is impossible in the
       case covered by \cite[Lemma~3.8]{expcombinatorics}.) 
      Indeed, assume without loss of generality that $j=0$, and let $p$ be the period of the  itinerary $\addm$.
      Then $p$ divides $n$; say $n=p\cdot q$. We consider the set $\mathcal{R}'$ obtained by adding to 
       $\mathcal{R}$ the orbit of $\s$ under
       $\sigma^p$.

     \begin{claim}[Claim] The set $\mathcal{R}'$ is mapped bijectively to itself
       under $\sigma^p$, and $\sigma^p$ preserves the cyclic ordering of addresses in $\mathcal{R}'$.
     \end{claim}
     \begin{subproof}
       We first prove that the claims hold for $\mathcal{R}$. Recall that all addresses in
        $\mathcal{R}$ have itinerary $\addm$,  which is an infinite sequence of integers (in particular,
        no address in  $\mathcal{R}$ is on the backward orbit of $\s$). 
        For every $j\geq 0$, all addresses of $\sigma^j(\mathcal{R})$ belong to the  interval
        $(m_j\s, (m_j+1)\s)$, and the shift map is injective on each  such interval and preserves the cyclic ordering.
        Furthermore, if $\addR\in \mathcal{R}$, then there  is a unique preimage of $\addR$ under $\sigma^p$
        whose  itinerary begins with $m_0 m_1 \dots m_{p-1}$, and which hence also belongs to  $\mathcal{R}$. 
        (Here we use the fact that $\s$ is periodic, and that hence $\addR$~-- whose itinerary is 
        an infinite sequence of integers by  assumption~-- is not on the orbit of $\s$ under the shift map.)

      Now $\sigma^n$ is clearly also bijective on  $\mathcal{R}'$. The reasoning that it 
        preserves the circular ordering is the same as above~-- here
        $\sigma^{p-1}(\mathcal{R}')$ also includes one exactly one of the 
        two boundary addresses of $(m_j\s,(m_j+1)\s)$, but this does not affect the argument.
      \end{subproof}

      Let $\tilde{n}$ be the minimal period of a periodic address $\addR$ 
        in $\mathcal{R}'$ (so $\tilde{n}$ is a multiple
        of  $m$ and divides $n$). Then $\sigma^{\tilde{n}}$ maps $\mathcal{R}'$ to itself, preserving 
        the cyclic ordering and fixing $\addR$. Let $\addt\in\mathcal{R}'$; then, in the cyclic order, either
         \begin{align*} 
            \addR &\leq \addt \leq \sigma^{\tilde n}(\addt) \leq  \sigma^{2\tilde n}(\addt) \leq
              \sigma^{3\tilde{n}}(\addt) \leq\dots \leq \addR, \qquad\text{or} \\
            \addR &\geq \addt \geq \sigma^{\tilde n}(\addt) \geq  \sigma^{2\tilde n}(\addt) \geq
              \sigma^{3\tilde{n}}(\addt) \geq \dots \geq \addR.\end{align*}
       So $\sigma^{j\tilde{n}}(\addt)$ is an eventually monotone, and hence convergent, sequence in $\Z^{\N_0}$. 
       But this is possible only if $\addt$ is itself periodic of period at most $\tilde{n}$. Hence all
       addresses in $\mathcal{R}$ are periodic of period $\tilde{n}=n$,  as claimed. Furthermore, the
       set of sequences of period $n$ in  $\mathcal{R}$ is clearly finite: for any $\addt\in\mathcal{R}$, we
       have $t_j \in \{m_j,m_j-1\}$ for  all  $j$, so there are at  most $2^{n}$. such addresses. (In  fact,
       as alluded to  in Remark 2 above, there are far fewer: either $\# \mathcal{R}= 0$ 
        or $\#\mathcal{R}=1$;  we do not require this fact).

      The case where $\addm$ is not eventually periodic is 
       Thurston's \emph{no wandering triangles theorem}, adapted to our context. 
       See e.g.\ \cite[Theorem~3.3]{schleicherfiberscompactsets} for a proof in the case of unicritical polynomials. 
       We are not aware of a published proof for exponential maps, and hence provide 
       the argument in Theorem~\ref{thm:NWT} below. 
 \end{proof}

 \subsection*{No wandering triangles}
    Let us begin by introducing some notation. 
  \begin{defn}[Wandering gaps]\label{defn:gaps}
    Two sets $A,B\subset\Sequb$ are called \emph{unlinked} if $A$ is completely contained in one of the intervals of
      $\Sequb\setminus B$ (and vice versa). 
     Let $A\subset\Sequ$, and consider the set obtained by shifting the initial entries of all addresses in $A$ by the 
     same
     integer $m\in\Z$. This set is called a \emph{translate} of $A$, and denoted  $m+A$.

     A \emph{wandering gap} (for the one-sided shift $\sigma$ on $\Z^{\N_0}$) is a set $A\subset\Z^{\N_0}$ with $\#\sigma^n(A) \geq 3$ for  all $n\geq 0$ such that the sets $\sigma^{n_1}(A)$ and $\sigma^{n_2}(A)$ are disjoint and  unlinked  for all non-negative $n_1\neq n_2$. 
     If $\#A=3$, then $A$ is also called a \emph{wandering triangle}.

   A wandering gap is \emph{of exponential combinatorial type} if the sets 
        $\sigma^{n_1}(A)+m_1$ and $\sigma^{n_2}(a)+m_2$ are disjoint and unlinked whenever
        $(n_1,m_1)\neq (n_2,m_2)$ (where  $n_1,n_2\geq 0$ and $m_1,m_2\in\Z$).
  \end{defn}

The condition of ``exponential combinatorial type'' is motivated 
precisely by the
     dynamics and combinatorics of exponential maps as discussed above;  see also 
     Observation~\ref{obs:unlinked}. Assuming this condition, we can
     now prove the absence of wandering gaps.  (In general, wandering gaps do exist
     for the shift map on any number $n\geq 3$ of  symbols \cite{wanderingtrianglesexist}, and 
     by the same reasoning they exist also for the shift on  infinitely many  symbols.)

\begin{thm}[No wandering triangles]\label{thm:NWT}
  There are no wandering gaps of exponential combinatorial  type.
\end{thm}
\begin{proof}
  Since every wandering gap contains a wandering triangle, it is enough to prove the absence of wandering triangles 
   of exponential combinatorial type;  hence the
    name of the theorem. 
   So suppose, by contradiction, that there is such a wandering triangle $T$, and define
      $T_n \defeq \sigma^n(T)$ for $n\geq 0$. Let $\mathcal{T}$ denote the set of all $T_n$ and their
      translates; recall that, by assumption, the elements of $\mathcal{T}$ are pairwise disjoint and pairwise unlinked. 
     For  ease  of  notation, we also write $\addR+m$ for the translate of an address $\addR$ by $m$; that is,
     \[ \addR+m = (r_0+m) r_1 r_2 r_3 \dots .\]

   \begin{claim}[Observation 1]
      If $T\in\mathcal{T}$ and
      $\addR^1,\addR^2 \in T$, then $\addR^2 < \addR^1+1$. 

     In particular, $\sigma|_{T}$ preserves the cyclic ordering of addresses, and 
      any two addresses in $T$ differ by at most $1$ from each other in every entry. 
  \end{claim}
  \begin{subproof}
    The first claim is immediate from the fact that $T$ is unliked with its translates, and implies the second claim (since $\sigma$ preserves
      cyclical order when restricted to the interval between an address and its translate by $1$). The final statement follows from the first, applied to 
      $\sigma^n(T)$ for all $n\geq 0$.
  \end{subproof}

  For two addresses $\addR^1,\addR^2$, let us write $N(\addR^1,\addR^2)$ for the index of the first entry in which
    these two addresses differ. 
    For any triangle $T\in\mathcal{T}$, also write $N(T)$ for the largest value of $N(\addR^1,\addR^2)$ for two different
     $\addR^1,\addR^2\in T$. 

  \begin{claim}[Observation 2]
    If $T\in\mathcal{T}$ is a triangle, then there is $\addt\in T$ such that $N(\addt,\addR) < N(T)$ for the other addresses
     $\addR\in T\setminus\{\addt\}$. 
   \end{claim}
   \begin{subproof}
      Let $T = \{\addR^1,\addR^2,\addt\}$, where $N(\addR^1,\addR^2)=N\defeq N(T)$. 
      By Observation~1, at position $N$, there are only two possible choices for the entry of each address in $T$. 
       In particular, $\addt$ does not differ from one of the two other addresses, say $\addR^1$, at position $N$. By definition of
        $N(T)$, we must have $N(\addt,\addR^1) < N$. Since $\addR^1$ and $\addR^2$ agree up to entry $N-1$, we also have 
         $N(\addt,\addR^1) = N(\addt,\addR^2)$. 
   \end{subproof}

   \begin{claim}[Observation 3]
     There are $n_2 > n_1 > 0$ such that 
       $N_1\defeq N(T_{n_1}) \geq N(T_{n_1-1})$ and $N_2\defeq N(T_{n_2}) \geq N(T_{n_1})$. 
   \end{claim}
   \begin{subproof}
     This is trivial if $N(T_n)$ is unbounded as $n\to\infty$. Otherwise, there are infinitely many $n$ for which 
         \[  N_1 \defeq N(T_n) = \limsup_{k\to \infty} N(T_k), \]
      and we choose two of these that are sufficiently large to ensure that $N(T_n) \leq N_1$ for all $n\geq n_1-1$. 
      (Note that, in this case, $N_1=N_2$.)
   \end{subproof} 

  Let $n_1$ and $n_2$ be as in Observation 3; we may assume furthermore that $n_2>n_1$ is chosen to be minimal. 
    Let $\tilde{T}_1$ and $\tilde{T}_2$ be translates of $T_{n_1-1}$ and $T_{n_2-1}$, chosen to contain
     an address having initial entry $0$, but no address with initial entry $-1$. (This is possible by Observation 1.) In particular,
     $\sigma(\tilde{T}_j) = T_{n_j}$ for $j=1,2$, and 
     \begin{equation}\label{eqn:N2estimate}
       N(\tilde{T}_2) = N(T_{n_2-1}) \leq N_1
     \end{equation} by minimality of $n_2$. 

  \begin{claim}[Observation 4]
   For $j=1,2$, the triangle $\tilde{T}_j$ contains two addresses
     $\addR^{j,0}$ and $\addR^{j,1}$ such that $r^{j,\ell}_0 = \ell$ for $\ell=0,1$, and such that
          $N(\addR^{j,1} , \addR^{j,0}+1) = N(\addR^{j,0}, \addR^{j,1}-1) = N_j+1$. 

   Furthermore, the third address $\addt\in \tilde{T}_j$ differs from one of these two addresses at position $0$, and from the other
    in an entry at a position $\leq N_j$. 
  \end{claim}
 \begin{subproof}
   If two addresses $\addt^0\neq \addt^1$ differ first in the $N$-th entry, and have $t^0_0 = t^1_0$, then 
      their images under the shift map clearly differ first in the $(N-1)$-th entry, and thus 
       \[ N(\sigma(\addt^0),\sigma(\addt^1)) = N(\addt^0,\addt^1) - 1. \]

   Now let $\addt^1 < \addt^0$ be the two addresses in $T_{n_j}$ with
      $N(\addt^1,\addt^0)= N(T_{n_j})$, and let $\addR^{j,1}$ and $\addR^{j,0}$ be their preimages in $\tilde{T}_j$. 
     Then 
      \[ N(\addR^{j,1} , \addR^{j,0}) \leq N(\tilde{T}_j) \leq N(T_{n_j}) = N(\addt^1 , \addt^0), \]
      and hence $\addR^{j,1}$ and $\addR^{j,0}$ differ in their initial entry. 

    It follows that $\addR^{j,\ell}$ begins with the entry $\ell$ for $\ell=0,1$, and that the two addresses agree in the next
      $N_j$ entries, proving the first claim. 

    The second claim follows from Observation 2 in a similar manner. 
 \end{subproof} 

  We can now reach the desired contradiction. Indeed, by Observation 1, the two triangles $\tilde{T}_1$ and $\tilde{T}_2$ are unlinked. 
     Since $N_2 \geq N_1$, it follows from Observation 4 that $\tilde{T}_2$ cannot belong to one of the bounded intervals of 
      $\Z^{\N_0} \setminus \tilde{T}_1$, and thus 
     $\tilde{T}_{2}$ is contained in the union of the two intervals $(\addR^{1,1}-1, \addR^{1,0})$ and $(\addR^{1,1},\addR^{1,0}+1)$. In particular, there must be two
     addresses in      $\tilde{T}_{2}$ that belong to the same of these two intervals, and hence agree in the first $N_1+1$ entries.
     So  $N(\tilde{T}_{2}) > N_1$, which contradicts~\eqref{eqn:N2estimate}. 
 \end{proof} 
\begin{remark}
  The above proof follows similar ideas as the usual proof of the wandering triangles theorem in the polynomial case, with one notable
   exception. The first step in the latter proof is usually to observe that the smallest side in the iterates of a wandering triangle must become
   arbitrarily small (in our terminology, this would mean that $N(T_n)\to\infty$ as $n\to\infty$). In the polynomial case, this follows from an
   area argument \cite[Lemma~2.11]{schleicherfiberscompactsets} which completely breaks down in the infinite-symbol case. This forces us to argue more carefully,
   observing that Observation~3 holds and is sufficient to conclude the proof, even if $N(T_n)$ remained bounded. 
   It is not clear whether this consideration  
      is an artifact of the proof or a genuinely new phenomenon in the transcendental
     case. That is, for families of transcendental entire functions
      where wandering triangles do exist, do we always have $N(T_n)\to\infty$?
\end{remark}

The following observation links the notions of gaps, and Theorem~\ref{thm:NWT}, with  
   itineraries of exponential maps.

\begin{obs}[Sets with different itineraries are unlinked]\label{obs:unlinked}
   Let $\s\in \Sequ$, and let $A,B\subset \Sequ$ be such that all addresses in $A$ share the same itinerary $\addu^A$, and all
      addresses in $B$ share the same itinerary $\addu^B$ (with respect to $\s$), but $\addu^A\neq \addu^B$. 
     Then $A$ and $B$ are unlinked.

  In particular, if $\#A\geq 3$, then $\addu^A$ is periodic.
\end{obs}
\begin{remark}
The final statement completes the proof of Proposition~\ref{prop:finitenumber}.
\end{remark}
 \begin{proof}
    The first claim is trivial when the initial entries of the two itineraries differ (by definition of itineraries). Furthermore, the shift map, when restricted
     to any of the intervals used in the definition of itineraries, preserves the circular order
     of external addresses. The fact that the two sets are unlinked hence follows by induction. 

   The second claim follows from the first, together with the No Wandering Triangles theorem.   
     Indeed, by the definition of itineraries, $\sigma^n(A)$ is disjoint from its
    translates for all $n\geq 0$. If $\addu^A$ was aperiodic but $\#A \geq 3$, then  $A$ would be a wandering gap, 
     contradicting Theorem~\ref{thm:NWT}.
 \end{proof}

 \begin{proof}[Proof of Theorem~\ref{thm:nowanderingtriods}]
    Suppose, by contradiction, that $z_0$ is as in Theorem~\ref{thm:nowanderingtriods}, but that $z_0$ is not 
      eventually periodic. By Theorem~\ref{thm:rays} (and since $f_a^n$ is locally injective near $z_0$), there is $n\geq 0$ such that
     $f^n(z_0)$ is the landing point of at least three dynamic rays, at addresses $\addR^1,\addR^2,\addR^3$. These three addresses form
     a wandering triangle $T$.
     Since the three rays associated to any translate of any iterate of $T$ also land together (due to  the
       periodicity of  the exponential map), the triangle is
       of exponential combinatorial type. This
      contradicts Theorem~\ref{thm:NWT}. 
\end{proof}

 \subsection*{Attracting and parabolic examples}
  Recall that Theorem~\ref{thm:endpoints} claims that, if $f_a$ is an exponential map such that $a\in F(f_a)$, then $E(f_a)$ is totally 
    separated. This is a consequence of the following result. 

     \begin{prop}[Attracting and periodic parameters]\label{prop:fatou}
       Suppose that $f_a$ is an exponential map with an attracting or parabolic periodic orbit. Then there exists a curve $\gamma\subset F(f_a)$ in the
        Fatou set such that no two endpoints of $f_a$ have the same itinerary with respect to $f^{-1}(\gamma)$. 
     \end{prop}
     \begin{proof}
       This follows from the fact that, under the given hypotheses, all dynamic rays of $f_a$ land, and two dynamic rays land at the same point 
        if and only if they have the same itinerary with respect to $f^{-1}(\gamma)$, where $\gamma$ is certain curve connecting
        the singular value to $\infty$, constructed precisely
         as described above. 
        (Note that, in the case where $f_a$ has an invariant Fatou component, we must now use a curve $\gamma$ tending to $-\infty$ and having
          $\extaddr(\gamma)=\infty$. It is then clear that any two dynamic rays at different addresses have different itineraries with respect to 
          $f^{-1}(\gamma)$, and hence the claim is trivial in this case, as mentioned above.) 

      In the case where $f_a$ has an attracting periodic orbit, the above claims are proved in
        \cite[Proposition~9.2]{topescaping} as a consequence of the stronger Theorem~9.1 in the same paper. It is also remarked at the 
        end of \cite[Section~9]{topescaping} that these results remain true for parabolic parameters, although the details are not given. 

      We note that the proof of our proposition in the parabolic (and also in the attracting) case can be achieved with considerably less
        effort than the results of \cite[Section~9]{topescaping}. 
        Indeed, by Lemma~\ref{lem:sharing}, the claim is trivial for endpoints of rays at unbounded external addresses. 
        Furthermore, for eventually periodic addresses (and hence for periodic itineraries, by Proposition~\ref{prop:finitenumber}), the claim is
        proved in \cite[Proposition~4.5]{expper}. 

      So let $g_{\addR^1}$ and $g_{\addR^2}$ be two rays at bounded addresses, sharing the same (bounded) itinerary $\addm$ which 
        is not eventually periodic. (Note that, in particular, $\sigma^j(\addm)$ is distinct from the itinerary of the parabolic orbit of $f_a$, for all $j\geq 0$.)
        We must show that, if both rays 
        land, they land at the same point in $\C$. 

       It follows from a simple hyperbolic contraction argument (similar to that in the proof of \cite[Proposition~9.2]{topescaping}) that, for
         any bounded itinerary which does not eventually agree with that of the parabolic orbit, there is a unique point $z_0$ with a bounded orbit
         which has this itinerary. Furthermore, any other point with the same itinerary must tend to infinity under iteration. 
         Because the landing points of $g_{\addR^1}$ and $g_{\addR^2}$ cannot be escaping (only rays at fast addresses land at escaping points), 
         it follows that both rays land at $z_0$. This completes the proof. 
     \end{proof}

\begin{proof}[Proof of Theorem~\ref{thm:endpoints}]
  If $f_a$ is an exponential map with $a\in F(f_a)$, then (as mentioned above), $f_a$ has an attracting or parabolic periodic orbit.
    It follows from Proposition~\ref{prop:fatou}, together with Observation~\ref{obs:itineraryseparation}, that 
     $E(f_a)$ is totally separated.

   On the other hand, suppose that $f_a$ is an exponential map such that the singular value $a$ is an endpoint or on a hair, and let 
     $\gamma\subset g_{\s}$ be the curve connecting $a$ to $\infty$, as above. Assume furthermore (as in the hypothesis of the theorem)
    that the kneading sequence
     $\K(\s)$ is not periodic; we must show that $E(f_a)$ is totally disconnected.

    So let 
     $X$ be a connected component of $E(f_a)$. Then all points of $X$ share the same itinerary $\addm$ with respect to
     $f_a^{-1}(\gamma)$. By passing to a forward iterate, we may assume that $\sigma^j(\addm)\neq \K(\s)$ for all $j\geq 0$. In particular,
     each $z\in X$ is the landing point of some dynamic ray $g_{\addR}$, where $\itin_{\s}(\addR)=\addm$. By Proposition~\ref{prop:finitenumber},
      the set of possible such addresses is finite, and hence $X$ is finite. Since $X$ is connected, it follows that it consists of a single point, as desired. 
 \end{proof} 

\section{Beyond the exponential family}\label{sec:general}

 In order to prove Theorem~\ref{thm:general}, we use similar ideas as in the proof of Theorem~\ref{thm:escapingendpoints}, together with
    the results of~\cite{rrrs} and~\cite{brush}, which establish the existence of Cantor bouquets in the Julia sets of the functions under consideration.
    The difference between the general case and that 
     of exponential maps is that we do not have such explicit information about the position and behaviour of rays as we did
    through our model $\F$. In particular, there may not be an analogue of the characterisation of ``fast'' addresses in
     terms of a simple growth condition. This means that, for some escaping endpoints, our argument would no longer yield a corresponding
    brush where all endpoints are escaping (i.e., there is no analog of Observation~\ref{obs:potentials}). Furthermore, we do not have as precise control over the positions of endpoints as we utilised in 
    the proof of Theorem~\ref{thm:Fexplodes}.

  However, both concerns are easily taken care of by using more general arguments: We do not need to find a sub-brush for \emph{every} fast
   address, but it is enough to exhibit \emph{one} Lelek fan with the desired properties. Similarly, we can use softer arguments to 
   establish the density of endpoints. In the following, we shall use the terminology of~\cite{brush}, and refer to that paper for
   further details and background.

\begin{rmk}[Escaping endpoints]\label{rmk:endpoints}
   Let $f$ satisfy the hypotheses of Theorem~\ref{thm:general}. Then there is again a 
     natural notion of ``escaping endpoints'': These are 
     those points $z_0\in I(f)$ with the following property. Let $n_0\geq 0$ be any number that is
     sufficiently  large to  ensure
     that the external address of $f^{n_0}(z_0)$ is defined in  the sense  of \cite{rrrs}. If
      $z\neq f^{n_0}(z_0)$ is an escaping point with the same external address, then
        $|f^n(z)| > |f^{n+n_0}(z_0)|$ for all sufficiently large $n$.

  If no critical point of $f$ has an escaping orbit, then it is plausible that,  again,
   no escaping endpoint belongs to the
    interior of an arc in the escaping set, and hence that there is an analogue of Definition~\ref{defn:endpoints} that yields the same notion. 
   However, if $f$ has an escaping critical point that is mapped to an escaping endpoint, then we can combine two preimage branches 
   of a curve in $I(f)$ that join at the critical point, and the two definitions no longer coincide. 
\end{rmk}

\begin{proof}[Proof of Theorem~\ref{thm:general}]
   Let $f$ satisfy the hypotheses of the theorem, and fix a sufficiently large number $Q>0$ (see below). 
     Since the set $S(f)$ of singular values of $f$ is bounded, we can assume without loss of
    generality that $S(f)\subset \D\cap f(\D)$. Set $W\defeq \C\setminus\overline{\D}$, and let $T$ be a component of $f^{-1}(W)$. Then
    $f\colon T\to W$ is a universal covering map. In particular, it follows that there is a sequence $(T_m)_{m=0}^{\infty}$ of pairwise disjoint
   components of 
    $f^{-1}(T)$ such that
   \begin{enumerate}[(i)]
      \item $f\colon T_m \to T$ is a conformal isomorphism for each $m$.
      \item $\inf_{z\in T_m} |z| > Q$ for all $m$, and $\inf_{z\in T_m} |z| \to\infty$ as $m\to\infty$.
      \item \label{item:closeness}
        For any $\zeta\in \bigcup_{j=0}^{\infty} T_j$ and any $m\geq 0$, set $\zeta_m \defeq (f_{|T_m})^{-1}(\zeta)$. 
           Then there is a constant $C>0$, independent of $\zeta$ and $m$,  such that
         \[ | \zeta_m - \zeta_{m+1}| \leq C\cdot |\zeta_m|. \]
   \end{enumerate}
   (The final point follows from the fact that the tracts can be chosen such that $\zeta_m$ and $\zeta_{m+1}$ have bounded hyperbolic distance
     from each other in $T$; compare the ``Claim'' in the proof of \cite[Lemma~5.1]{boettcher}.)

   Consider the set 
      \[ X \defeq \{z\in \C\colon f^k(z) \in \bigcup_{m=k}^{\infty} T_m \text{ for all $k\geq 0$}\} \subset I(f). \]
    Given a sequence $\s\in {\N_0}^{\N_0}$ with $s_k\geq k$ for all $k$, we also define
     \[ X_{\s} \defeq \{ z\in X\colon f^k(z) \in T_{s_k} \text{ for all $k\geq 0$}\}. \]
    It follows from the definitions that $X_{\s}\cup\{\infty\}$ is a compact, connected set. Note that 
      $X_{\s}\subset X \subset I(f) \subset J(f)$ (where the final inclusion is \cite[Theorem~1]{eremenko_lyubich_2}. 
      Furthermore, $X_{\s}$ is either empty or an arc connecting a finite endpoint
     $e(\s)$ to infinity, and $e(\s)$ is an escaping endpoint in the sense of Remark~\ref{rmk:endpoints} \cite{rrrs}. 
     In particular, if $z\in X_{\s}$, then $|f^n(z)| \geq |f^n(e(\s))|$ for all sufficiently large $n$. 

   If $Q>0$ was sufficiently large, then by~\cite[Theorem~1.6]{brush}, there is a Cantor Bouquet $X_0\subset J(f)$ containing all points
    whose orbits contain only points of modulus at least $Q$. Hence $X\subset X_0$, and $X\cup\{\infty\}$ is a smooth fan with top $\infty$.  

    \begin{claim}
     If $Q$ was chosen sufficiently large, then $X\cup\{\infty\}$ is a Lelek fan.
    \end{claim}
    \begin{subproof}
      We only need to prove that endpoints are dense. First observe that, by the definitions, $f(X)\subset X$. Furthermore, if 
       $z\in X$, then $z$ is an endpoint of $X$ if and only if $f(z)$ is. 
       
      Let $z\in X_{\s}$ for some address $\s$ as above, and let $n\geq 0$ be sufficiently large. Let $\zeta\defeq e(\sigma^{n+1}(\s)) = f^{n+1}(e(\s))$. 
       Then $|f^n(z)| \geq |\zeta_{s_n}|=f^n(e(\s))$, 
           provided that $n$ was chosen sufficiently large. By~\ref{item:closeness}, there is thus some $m_n\geq s_n$ such that 
             \begin{equation}\label{eqn:choiceofm_n} |\zeta_{m_n}| \leq |f^n(z)| \leq C\cdot |\zeta_{m_n}|. \end{equation}
       Define an address $\s^n$ by
        \[ \s^n_k \defeq \begin{cases} s_k & \text{ if $k\neq n$} \\ m_n & \text{otherwise}.\end{cases} \] 
       It follows from~\cite[Lemma~1]{eremenko_lyubich_2} that
        \[ |f'(z)| \cdot \frac{|z|}{|f(z)|} \to \infty \]
        as $|f(z)|\to\infty$; i.e., $f$ is strongly expanding with respect to the cylindrical metric
          $|dz|/|z|$ when $|f(z)|$ is large. The cylindrical distance between $\zeta_{m_n}$ and $f^n(z)$ remains bounded as
         $n\to\infty$ by~\eqref{eqn:choiceofm_n}. This implies that 
         $e(\s^n)\to z$, as desired.
   \end{subproof}
         
  In particular, the set $A$ of endpoints of $X$ has the property that $A\cup \{\infty\}$ is connected. The first part of the theorem follows
    using Lemma~\ref{lem:pullbacks}, in the same manner as in the proof of Theorem~\ref{thm:endpoints}. 

   Suppose $f$ is hyperbolic,  and let $|\lambda|$ be sufficiently small. Then the function 
       $f_{\lambda}\defeq \lambda f$ is hyperbolic with connected Fatou set; see 
        \cite[\S5,\ p.\ 261]{boettcher}. As both  maps are  hyperbolic, it follows from \cite[Theorem~1.4]{boettcher} that the escaping set $I(f)$ is 
      homeomorphic to the escaping set $I(f_{\lambda})$, and that $f$ and $f_{\lambda}$ are conjugate
      on  these sets. Furthermore,
      the conjugacy maps escaping
     endpoints of $f$ to escaping  endpoints of $f_{\lambda}$.
     (Indeed, hyperbolic functions have no
      escaping  critical points.  Hence,  as  discussed in  Remark~\ref{rmk:endpoints}, escaping endpoints can
      be defined in terms of  the topology  of  the set of all escaping points, and hence any homeomorphism between
      two escaping sets of hyperbolic functions must map escaping endpoints again to escaping endpoints. 
      Alternatively, the claim can be deduced directly from the proof  of  \cite[Theorem~1.4]{boettcher}.) 
      On the other hand, the Julia set of $f_{\lambda}$ is 
     a Cantor Bouquet by \cite[Theorem~1.5]{brush}, and hence its set of endpoints is totally separated. 

  (We remark that Theorems~1.5 and~1.6 of \cite{brush}, which we used in this proof, were stated only for 
    functions that can be written as the composition of finite-order functions with bounded singular sets. However, by
    Corollaries~6.3 and~7.5 of the same paper, they hold more generally when $f$~-- or, more precisely, one of the logarithmic transforms of $f$~-- satisfies a 
    uniform head-start condition in the sense of \cite{rrrs}.)
\end{proof}

\begin{rmk}[Escaping endpoints in the strong sense]
  We remark that 
   Theorem~\ref{thm:general} likely still holds if an ``endpoint'' is defined analogously to Definition~\ref{defn:endpoints};
    to avoid confusion, let us refer to these as \emph{escaping endpoints in the strong sense}. Let $\tilde{A}$ be obtained
    from the set $A$ constructed in the proof above by removing the grand orbits of all critical points. 

  It is easy to see that removing a countable set from the set of endpoints of a Lelek fan does not
    affect its explosion point property. (If $X$ is a straight brush, and $(\alpha_n)_{n=1}^{\infty}$ is a sequence of irrational numbers, then 
    choose small intervals $I_n$ with rational endpoints around each $\alpha_n$. If we remove all those points from $X$ whose height belongs to 
    one of the $I_n$, the result (if nonempty) will still be a Cantor bouquet, and hence its set of endpoints has the explosion point property.
    By taking the union of the endpoints of these bouquets for all possible choices of $I_n$, we obtain all endpoints apart from those at heights
     $\alpha_n$, as desired.) 
    Hence $\tilde{A}\cup\{\infty\}$ is also connected. Suppose furthermore that no ``slow'' ray~-- i.e.\ one on which the iterates do not escape to infinity
     uniformly~-- can land at an escaping endpoint. (It seems plausible that this is always true, but we are not aware of a proof; the argument for exponential maps
     relies on global periodicity and breaks down in general.) Then, as alluded to in Remark~\ref{rmk:endpoints},
     every point that eventually maps to $\tilde{A}$ will be an escaping endpoint in the strong sense.
     Hence the set of all escaping endpoints in the strong sense, together with infinity, is also connected.
\end{rmk}

\begin{rmk}
  The result in \cite[Theorem~1.4]{boettcher} has been extended to a large class of ``strongly subhyperbolic'' entire functions by Mihaljevi{\'c}-Brandt
     \cite{helenaorbifolds}, and hence $\infty$ is an explosion point for $\tilde{E}(f)$ also for these functions. It seems likely that 
     similar combinatorial techniques as those used in Section~\ref{sec:accessible} can be extended to entire functions
     as in Theorem~\ref{thm:general}, 
     assuming that the set of singular values is finite, and that singular values satisfy similar accessibility conditions from the escaping set. 
     In particular, we believe that $\infty$ is an explosion point for $\tilde{E}(f)$ whenever $f$ is \emph{geometrically finite} in the sense of
     \cite{helenalanding}; see also \cite[Chapter~5]{helenathesis} for a discussion of itineraries for geometrically finite functions. However, we shall
     not pursue this question here. 

    On the other hand, in general 
      the set of \emph{all} endpoints will not be totally separated, or even
       disconnected. Indeed, for the map $f(z) = \pi\sin(z)$, 
      every point in $\R$ is an endpoint \cite{dierkcosine}. It is easy to 
      see 
      that the union of iterated preimages of $\R$ is a connected set that is 
      dense in the plane, and hence $E(f)$ is connected. 

    Finally, we note also that Theorem~\ref{thm:denseray}, as stated, does not extend to the functions covered by Theorem~\ref{thm:general}. Indeed,
     the escaping set of the function $z\mapsto (\cosh(z))^2$ has a dense path-connected component not containing any endpoints. This component
     is obtained by taking all iterated preimages of the real axis; see \cite[pp.~804--805]{ripponstallardfast}. On the other hand (since the real axis contains
     no endpoints) the set of endpoints is disconnected, and indeed it seems likely that it is totally separated. 
     The reason that the proof of Theorem~\ref{thm:denseray} breaks down here is that, for maps with escaping critical points, it is no longer true
     that the escaping set locally contains the structure of a Cantor bouquet. 
\end{rmk}

\section{Do escaping endpoints explode in general?}\label{sec:further}

 In this section, we discuss evidence for Conjecture~\ref{conj:endpoints}. Let us begin by noting that it appears unlikely that the
   hypotheses of Theorem~\ref{thm:escapingendpointsexplode} hold for all $a\in\C$. Indeed, suppose that $f$ is a quadratic polynomial having
   a Cremer point (i.e., a non-linearisable periodic point of multiplier $e^{2\pi i\alpha}$, with $\alpha\in\R$ irrational). Then Kiwi~%
    \cite[Corollary~1.2]{kiwicremer} has
    proved that the 
   critical point of $f$ is not accessible from the basin of infinity. It seems likely that, likewise, for an exponential map $f_a$ with a Cremer point, the
   singular value $a$ is never the endpoint of a dynamic ray. Similarly, Perez-Marco~\cite{perezmarcohedgehogs} proved (according to Kiwi) the existence
   of \emph{infinitely renormalisable} quadratic polynomials with inaccessible critical points; here all periodic
    cycles are repelling. Again, it is plausible that such examples exist for exponential maps also. 
  Let us show how we can strengthen Theorem~\ref{thm:escapingendpointsexplode} to cover the latter 
   case.
  To do so, we rely on similar ideas as before, but utilise a more detailed understanding of exponential 
  combinatorics to remove the condition that the singular value is an endpoint or on a hair.
   We shall from now on assume
    familiarity with deeper results concerning the 
   combinatorial
   structure of exponential parameter space, as established in~\cite{landing2new,expcombinatorics,expbifurcationlocus,bifurcations_new}.

  \begin{thm}[More escaping endpoints explode]\label{thm:ghostlimbs}
    Let $a\in\C$ be a parameter such that $f_a$ does not have an irrationally indifferent periodic orbit, and such that all repelling periodic points of $f_a$ are 
      landing points of periodic dynamic rays. Then $\infty$ is an explosion point for $\tilde{E}(f_a)$. 
  \end{thm}
\begin{proof}[Sketch of  proof]
  If $a\in F(f_a)$, then the result is covered by
    Theorem~\ref{thm:escapingendpointsexplode}. So  we may assume that $a\in J(f_a)$. Likewise, Theorem~\ref{thm:escapingendpointsexplode} also applies
    when $a$ is an escaping point of $f_a$ (as then $a$ is either on a hair or an escaping endpoint),  so we may assume $a\notin I(f_a)$.

  Even without knowing that there is a ray landing at $a$,  we can 
    still associate one or  more external addresses $\s$, and a kneading
    sequence $\K(\s)$, to $a$ as follows.
    In \cite{expbifurcationlocus}, the notion of the \emph{extended fibre} $\breve{Y}$ of such a parameter was defined.
    This set $\breve{Y}$ is a closed subset of parameter space containing a finite (and non-zero) 
    number of \emph{parameter rays}, with associated 
    exponentially bounded external addresses \cite[Lemma~17]{expbifurcationlocus}. 
    All of these addresses share the same kneading sequence,
    which it hence makes sense to refer to as the kneading sequence $\K(a)$ of the parameter $a$. 
   (See \cite[Appendix~A]{expcombinatorics}.) The theory of \emph{internal addresses} allows us to
    translate the kneading sequence into information about the position  of  $a$ in  parameter space, and vice versa.
    In particular, the kneading sequence of a parameter $a$ is  periodic if and only  if  there is a hyperbolic
    component $U$ in  parameter space (i.e., an  open connected region in which all  parameters have an  attracting 
    periodic orbit) such that $a$ cannot be separated from $U$ by a pair of periodic parameter rays 
    landing at a common  point. As discussed in \cite[Theorem~4]{landing2new}, for any such  parameter, there
    is a cycle of periodic points that are not landing points of periodic dynamic rays. Hence the  assumption  that
    all periodic points of  $f_a$ are repelling, and that each of these is the landing point of a periodic ray,
    implies that the kneading sequence $\K(a)$ is aperiodic.

  Now let $S$ be the finite set of exponentially   bounded addresses associated  to $a$ as above; let us assume first for simplicity that $\#S\leq 2$. 
    In ``nice'' cases, we would expect the dynamical ray $g_{\s}$ of $f_a$,
    for $\s\in S$, to land at the singular value $a$, but as discussed above this does not hold in general. However, the construction of fibers ensures that $g_{\s}$ cannot be
    separated from $a$ in a certain sense. More precisely, if two periodic or preperiodic dynamical rays land at a common point, let us call their union together with the landing point
    a ``dynamical separation line''. The singular value $a$ cannot be separated from $g_{\s}$, for $\s\in S$, by such
    a separation line.

\begin{claim}
  Any dynamic ray $g_{\s}$, where $\s$ is exponentially bounded and $\s\notin S$, can be separated from $a$ by a dynamical separation line. Moreover, the addresses 
      of the rays in these separation lines are uniformly exponentially bounded (independently of $\s$).
\end{claim}
\begin{subproof}
Let $S=\{\s^-,\s^+\}$, with $\s^-\leq \s^+$. Then the characteristic addresses $(\addR^{j-},\addR^{j+})_{j\geq 1}$ 
    of the hyperbolic components appearing in the internal address of
    $a$ (see \cite{expcombinatorics}) satisfy 
     \[ \addR^{j-} < \addR^{(j+1)-} < \s^- \leq \s^+ < \addR^{(j+1)+} < \addR^{j+} \]
    for all $j$, with $\addR^{j-}\to \s^-$ and $\addR^{j+}\to \s^+$. (Observe that we use the fact that the kneading sequence is aperiodic, and hence the 
     internal address is infinite.) Furthermore, the rays at addresses $\addR^{j-}$ and $\addR^{j+}$ have a common periodic landing point, and hence the corresponding separation
     line separates $a$ from $g_{\s}$ for all $\s\notin [s^-,\s^+]$. Moreover, by the internal address algorithm, the kneading sequences of the 
     addresses $\addR^{j\pm}$ are all exponentially bounded with the same bound as $\K(a)$, and hence the addresses involved are uniformly exponentially bounded. 
     In particular, if $\s^-=\s^+$, then the claim is proved.

  If $\s^-<\s^+$, so that $\# S = 2$, we claim that there is also a sequence of preperiodic separation lines with associated external addresses
      $\{\tilde{\addR}^{j-},\tilde{\addR}^{j+}\} \subset (\s^-,\s^+)$, converging to $\{\s^-,\s^+\}$ from the inside. It is possible to deduce this using the methods of
     \cite{expcombinatorics}. Instead, we notice that, in this case, the addresses in $S$ must be bounded \cite[Theorem~A.3]{expcombinatorics}, and hence the claim 
     reduces to the corresponding statement for unicritical polynomials, using the same reduction as in the proof of \cite[Theorem~A.3]{expcombinatorics}. But in this setting, the claim
      is well-known: For a unicritical polynomial without nonrepelling cycles, the minor leaf of the corresponding lamination is approximated on from both sides by periodic or pre-periodic leaves.
\end{subproof}

  Observe that a separation line does not contain any escaping endpoints. Now suppose that $\addt^1$ and 
      $\addt^2$ are exponentially bounded  addresses whose itineraries (with respect to $s\in  S$) are defined and
      do not coincide. 
      Then it follows that  $g_{\addt^1}$ and $g_{\addt^2}$ 
      can be separated by a separation line. Indeed, suppose first that the initial itinerary entries differ,  and consider
      a separation line $\gamma$ that, separates both $g_{\sigma(\addt^1)}$ 
       from $a$. Then the component of $f_a^{-1}(\C\setminus \gamma)$ containing $g_{\addt^1}$ can contain
      only rays whose initial itinerary entry agrees with that of $\addt^1$. This implies that $\addt^1$ and $\addt^2$
      are indeed separated.  The case where the itineraries of $\addt^1$ and $\addt^2$ first differ in  the
       $n$-th entry,  for $n>0$, follows inductively by taking pullbacks. 
      Given  that any two endpoints have different itineraries by 
      Proposition~\ref{prop:slowsharingitineraries}, it follows again 
      that the set of escaping endpoints is
      totally separated.

If $\#S \geq 3$, the proof proceeds analogously, using separation lines lying in each of the complementary intervals of $S$. Alternatively, note that in this case
    the addresses in $S$ are eventually periodic \cite[Theorem A.3]{expcombinatorics}, and it follows from \cite{beninimisiurewiczfibers} that the singular value $a$ is
    preperiodic. In this case, it is known that $a$ is an endpoint \cite[Theorem~4.3]{expper}, and hence the result follows from Theorem~\ref{thm:escapingendpointsexplode}.
\end{proof}

  Conjecturally, every repelling periodic point of an exponential map with non-escaping singular value is the landing point of a periodic dynamic ray. 
    Hence Theorem~\ref{thm:ghostlimbs} provides strong evidence that 
    Conjecture~\ref{conj:endpoints} is valid at least for all exponential maps without Cremer points or Siegel discs. 

  Also, if $f_a$ has a Siegel disc whose boundary contains the singular value $a$, then conjecturally the singular value $a$ is an endpoint. If this is the
   case, then the claim in Conjecture~\ref{conj:endpoints} holds for such parameters by Theorem~\ref{thm:escapingendpointsexplode}. 
   However, when 
    $f_a$ has a Cremer point (or, indeed, a Siegel disc whose boundary does not contain the singular value), then in view of the results discussed above it is likely 
    that $a$ is \emph{not} an endpoint. 

   More recent results on quadratic polynomials with Cremer fixed points (and Siegel discs) suggest that Conjecture~\ref{conj:endpoints} is true in these
    cases also. Indeed, for a large class of such maps, Shishikura has announced the existence of an arc in the Julia set connecting the fixed point and
    the critical point, and although the general case is out of reach, this is now widely believed to be true in general. 
    For exponential maps with an irrationally indifferent orbit, it seems likely that there is similarly
     an arc of non-escaping points connecting the singular value to this orbit. We could then join this arc with one of its preimages to obtain a curve connecting the singular value to $\infty$, and apply the
   results of Section~\ref{sec:accessible}. 

 Finally, let us briefly remark that our results for the set of escaping endpoints (and all endpoints) leave open similar questions for the case of 
   \emph{non-escaping} endpoints, even for the case of $e^z-a$ with $a>1$. Evdoridou \cite{vassoweb} has recently announced a result
   that implies that the set of non-escaping points of the \emph{Fatou function} $z\mapsto z+1+e^{-z}$ (whose Julia set is a Cantor bouquet), together with
    $\infty$, is totally disconnected. 
    The proof can be adopted to show also that the set of non-escaping endpoints~-- and even the set of all endpoints not belonging to the 
    fast escaping set $A(f)$~-- of an exponential map with Cantor bouquet Julia set has the
    same property. This shows that our results are optimal in a certain sense.

\bibliographystyle{amsalpha}
\bibliography{biblio_explosion}

\end{document}